\documentclass[12pt]{article}
\usepackage[numbers,sort&compress]{natbib}
\usepackage{enumerate}
\usepackage{amscd}
\usepackage{amsmath}
\usepackage{latexsym}
\usepackage{amsfonts}
\usepackage{setspace}
\usepackage{amssymb}
\usepackage{amsthm}
\usepackage{amsmath}
\usepackage{verbatim}
\usepackage{mathrsfs}
\usepackage{enumerate}
\usepackage[hypertexnames=false]{hyperref}

\theoremstyle{definition}
\theoremstyle{definition}\newtheorem{theorem}{Theorem}[section]
\theoremstyle{plain}\newtheorem{lemma}[theorem]{Lemma}
\theoremstyle{plain}
\theoremstyle{plain}
\theoremstyle{definition}\newtheorem{remark}{Remark}[section]
\usepackage{xcolor}

\newcommand{\R}{\mathbb{R}}
\newcommand{\ii}{\mathrm{i}}
\newcommand{\e}{\mathrm{e}}
\newcommand{\diver}{\mathrm{div}}
\newcommand{\dd}{\mathrm{~d}}
\newcommand{\w}{\widehat}

\newcommand{\mr}{\mathbb{R}}
\newcommand{\lef}{\left\|}
\newcommand{\rig}{\right\|}
\newcommand{\define}{\stackrel{\mathrm{def}}{=}}
\allowdisplaybreaks

\numberwithin{equation}{section}
\begin{document}
	\title{Large time behavior of  solutions to the 2D damped wave-type magnetohydrodynamic equations}

\author{Yaowei Xie, Huan Yu}


	\author{Yaowei Xie\footnote{School of Mathematical Sciences, Capital Normal University, Beijing, 100048, P.R.China. Email: mathxyw@163.com},~~~\,\,\,\,Huan Yu\footnote{School of Applied Science, Beijing Information Science and Technology University, Beijing, 100192, P.R.China.  Email: huanyu@bistu.edu.cn}}
\date{}
\maketitle

\begin{abstract}
In this paper, we are concerned with  the 2D damped wave-type magnetohydrodynamic system (abbreviated as MHD-wave system). The purpose of  this paper is to study the large time  behavior of  solutions to  the MHD-wave system, espesically to investigate the influence of the bad term  $\gamma \partial_{tt}b$  on the large time behavior. Rates of decay  are obtained for both the solutions  and higher derivatives in different Sobolev spaces with  explicit rates of $\gamma$, which shows that  the decay rates  closely align with that of the MHD system under the same norm, for any fixed $\gamma>0$. In this sense, these decay  rates are optimal.
  \end{abstract}
	\section{Introduction}
	\subsection{Background and motivation}
	This paper examines the initial-value problem for  the 2D damped wave-type magnetohydrodynamic system (or simply MHD-wave system)
\begin{align}\label{mhdwave}
	\left\{\begin{array}{l}
		\partial_t u-\Delta u+u \cdot \nabla u+\nabla p=b\cdot \nabla b,~ x\in\mathbb{R}^2,t>0 \\[2mm]
	\gamma \partial_{tt}b+	\partial_t b-\Delta b+u\cdot \nabla b=b\cdot\nabla u, \\[2mm]
		\nabla\cdot u=\nabla \cdot b=0, \\[2mm]
		u(x, 0)=u_0(x),\,\, b(x, 0)=b_0(x),\,\, (\partial_t b)(x,0)=a_0(x),
	\end{array}\right.
\end{align}
where $u=u(x, t)$ denotes the  velocity field of the fluid, $b=b(x, t)$ the magnetic field and $p=p(x, t)$  the scalar pressure.  $\gamma>0$ is a real parameter.

The MHD-wave system  is formally derived from Maxwell's equations of electromagnetism by retaining the usually ignored small term involving the product of permittivity and magnetic permeability (see, e.g. \cite{wu-2022-wave,wave-2021-jde-fouriersobolev}).
When $\gamma=0$, the term $\gamma \partial_{tt}b$ disappears, system \eqref{mhdwave} reduces to the 2D MHD equations. In recent years, 
there have been substantial  developments on various fundamental
issues   concerning MHD systems (see e.g. \cite{abidi-2017-GlobalSolution3D,A,jiang-2021-decay,xie-2024-zamp,jiu-2015-MHD,lin-2015-GlobalSmallSolutions,lin-2014-GlobalSmallSolutions,panGlobalClassicalSolutions2018,ren2014global,schonbek-1996-optimal,sermange-1983,tan-2018-decay-bound,TYZ,WU1,WU2,zhang-2020-GlobalWellPosedness2D,xu-2015-GlobalSmallSolutions,YS}). However, among the existing literature regarding system \eqref{mhdwave}, very few has addressed the rigorous mathematical studies. 
Recently, the small data global well-posedness and singular limit  of system \eqref{mhdwave}
 in Fourier--Sobolev spaces have been obtained
 in \cite{wave-2021-jde-fouriersobolev}. 
Ji-Wu-Xu \cite{wu-2022-wave} proved the global well-posedness of system \eqref{mhdwave} in a critical Sobolev setting when $\gamma$  and the size of the initial data satisfy a suitable constraint. Moreover, they proved that the solution  converges to that
of the corresponding MHD system with an explicit rate.
Sun-Wang \cite{sun-2023-wave} considered the more dissipative fractional-order MHD-wave system and  obtained  global existence and uniqueness of solutions.

In view of the available literature, we find that there are not results on the asymptotic behavior of  solutions to the MHD-wave system. Indeed, in contrast to the standard 2D MHD equations, the equation
of $b$ in \eqref{mhdwave} is hyperbolic, and the extra term $\gamma \partial_{tt} b$ is mathematically a bad term in
the sense of energy estimates.
The goal of this paper is to study the large time  behavior of  solutions to  the MHD-wave system, espesically to investigate the influence of the bad term  $\gamma \partial_{tt}b$  on the large time behavior.


In the last subsection, we introduce some notations to be used.
\begin{itemize}
	\item For any $m>0$ and $1\leq c<2$, we set 
	\begin{align*}
		X^m(\mathbb{R}^2)&=\left(H^m(\mathbb{R}^2),H^{m+1}(\mr^2),H^m(\mr^2)\right).\\
		X^{m,c}(\mathbb{R}^2)&=\left(H^m(\mathbb{R}^2)\cap L^c(\mathbb{R}^2),H^{m+1}(\mathbb{R}^2)\cap L^c(\mr^2),H^m(\mathbb{R}^2)\cap L^c(\mr^2)\right).
	\end{align*}
	\item Let $X_1, X_2$ be Banach spaces,  \begin{align*}
		\| \cdot \|_{X_1\cap X_2} \define \|\cdot\|_{X_1}+\|\cdot\|_{X_2}.
	\end{align*}
	For arbitrary  functions $f$ and $g$,
	\begin{align*}
		\|f,g\|_{X_1}\define \|f\|_{X_1}+\|g\|_{X_1}.
	\end{align*}
	\item For $s\in\mathbb{R}$, $\Lambda^s f$ is defined as follows
	\begin{align*}
		\mathscr{F}\left(\Lambda^s f\right)(\xi)=|\xi|^s\mathscr{F} \left(f\right)(\xi),
	\end{align*}
	where the $\mathscr{F}$ is the Fourier transform, defined as follows
	\begin{align*}
		\mathscr{F}(f)(\xi)=\dfrac{1}{2\pi}\int_{\mr^2}\e^{-\ii x\cdot\xi}f(x)\dd x.
	\end{align*}
\end{itemize}

 \subsection{Main results}

We now proceed to  state our  main results.
Our first result is concerned with the global well-posedness of system \eqref{mhdwave} and the decay estimates of global strong solution in $L^q$-norm.
\begin{theorem}\label{thm1}
	 Let $m>0$ and $\gamma>0$. Assume that the  initial data $(u_0,b_0,a_0)\in X^m(\mathbb{R}^2)$ satisfying $\nabla\cdot u_0=\nabla \cdot b_0=\nabla\cdot a_0=0.$ If  there exists a suitable constant $\varepsilon> 0$, such that
\begin{align}\label{1.1}
	\left\|(u_0,b_0,a_0)\right\|_{X^m(\mr^2)}\leq \varepsilon,
\end{align}
then there exists a unique global solution $(u,b)$ to system \eqref{mhdwave} satisfying, for any $t>0$,
\begin{align}\nonumber
&\|u(t)\|_{H^m}^2+\| b(t)\|_{H^m}^2+2\gamma^2  \|\partial_t  b(t)\|_{H^m}^2+2\gamma\|\nabla b(t)\|_{H^m}^2\\\nonumber
&~~~~~~~+\int_0^t \|\nabla u(\tau)\|_{H^m}^2+\|\nabla b(\tau)\|_{H^m}^2+\gamma\|\partial_\tau b(\tau)\|_{H^m}^2 \dd\tau \\\label{th1-1}
&\leq C\left(1+\gamma^{\frac{2m}{2m+2}}\right)\left\| (u_0,b_0,a_0)\right\|_{X^m(\mr^2)}^2.
\end{align}
Furthermore, the global solution obeys the following $L^q$ decay estimates:
	\begin{align}\label{lqdecay}
	\|u(t)\|_{L^q}+	\|b(t)\|_{L^q}\leq C\left(1+\gamma\right)t^{\frac{1}{q}-\frac{1}{2}}\left\| (u_0,b_0,a_0)\right\|_{X^m(\mr^2)} ,
\end{align} for any  $q\geq 2$ and $t>0$.
\end{theorem}

Given additionally the $L^c$-norm, with $1\leq c<2$ smallness assumption for the initial data, we are able to
obtain higher-order derivative decay estimates in $L^2$ framework.  Specifically, we have the following theorem.

\begin{theorem}\label{thm2}
Assume that the divergence-free initial data $(u_0,b_0,a_0)\in X^{m,c}(\mr^2)$ with $m>0, 1\leq c<2.$
If there exists a positive constant $\varepsilon$, such that 
\begin{align*}
		\left\|(u_0,b_0,a_0)\right\|_{X^{m,c}(\mr^2)}\leq \varepsilon,
\end{align*}
then the  global solution $(u,b)$ of  system \eqref{mhdwave} established in Theorem \ref{thm1} satisfies, for any $t>0$,
\begin{align}\label{th2-1}
		\|\Lambda^\beta u(t)\|_{L^2}+\|\Lambda^\beta b(t)\|_{L^2}\leq C \left(1+\gamma^{1+\frac{1}{c}-\frac{1-\beta}{2}}\right)	 (1+t)^{\frac{1-\beta}{2}-\frac{1}{c}}\|(u_0,b_0,a_0)\|_{X^{m,c}(\mr^2)},
\end{align}
with $0\leq \beta \leq m$.
In addition, for any $t > 0$ and $0 \leq \varrho < m + 1$,
\begin{align}\label{th2-2}
	\|\Lambda^\varrho b(t)\|_{L^2}\leq C\left(1+\gamma^{1+\frac{1}{c}-\frac{1-\varrho}{2}}\right)(1+t)^{\frac{1-\varrho}{2}-\frac{1}{c}}\|(u_0,b_0,a_0)\|_{X^{m,c}(\mr^2)} .
\end{align}
\end{theorem}



\begin{remark}
The decay estimates of solutions to system \eqref{mhdwave} are obtained  with an explicit rate of $\gamma$, which shows that  the decay rates  closely align with that of the MHD system under the same norm, for any fixed $\gamma>0$. In this sense, the decay  estimates in Theorems \ref{thm1} and \ref{thm2} are optimal.
\end{remark}

\subsection{Challenges and methodology}

Due to the presence of the extra term $\gamma \partial_{tt}b $ in the MHD-wave equations, the direct energy method applied to 2D MHD equations no longer works for the MHD-wave equations. Instead, we make use of the following integral representations of the  MHD-wave system, which is first derived in \cite{wu-2022-wave}.
\begin{equation}\label{u-integral}
	\begin{split}
		u(x,t)&=\e^{t\Delta}u_0+\int_0^t \e^{(t-\tau)\Delta}\mathbb{P}\left(b\cdot\nabla b-u\cdot\nabla u\right)(\tau)\dd \tau.
	\end{split}
\end{equation}
\begin{equation}\label{b-integral}\begin{split}
		b(x,t)&=\left(K_0(t,D)+\frac{1}{2}K_1(t,D)\right)b_0+\gamma K_1(D)a_0\\&~~~~~~~~~+\int_{0}^{t}K_1(t-\tau,D)\left(b\cdot\nabla u-u\cdot\nabla b\right)(\tau)\dd \tau.\end{split}
\end{equation}

The symbols $\w{K_0}(\xi,t)$,  $\w{K_1}(\xi,t)$ of the above operators $K_0(t,D)$ and $K_1(t,D)$ behave differently and obey different upper bounds (see Lemma \ref{integral} in Section 2 for details). Indeed,  $K_0(t,D)$ and $K_1(t,D)$ reflects the effect of dissipation in the low frequency part, while  $K_0(t,D)$ and $K_1(t,D)$  demonstrates the effect of damping in the high frequency part.

Since $K_0(t,D)$ and $K_1(t,D)$  demonstrates the effect of damping in the high frequency part,  the higher-order derivatives generated by the convection term are difficult to be handled.
Hence, the main difficulty  in  dealing with the long-time behavior problem lies in the estimation of the nonlinear term at high-frequency of the magnetic equations. To overcome the difficulty, we use \eqref{fren-2}, which includes a negative power of frequency $|\xi|^{-\vartheta}$ and represents a negative derivative $\left|\nabla\right|^{-\vartheta}$ in physical space, allowing us to offset the higher derivative information from the convection term. However, it will gain  a  negative power $-\frac{\vartheta}{2}$ of $\gamma$, tending to infinity as $\gamma\to 0$. In order to show that the long-time behavior of the solution to system \eqref{mhdwave}  is coincide with that to 2D MHD equtaions, it is required to discover additional positive powers of $\gamma$ from other perspectives, which is summarized in Lemma \ref{expintegral}.

In addition, to show  decay estimate of the higher  derivatives of $b$ in $H^\varrho (m<\varrho<m+1)$,  we need to find ways and means to put  more order  derivatives on $b$, since we do not have information about the higher order derivatives of $u$.


Now, we present a detailed explanation of the difficulties casused by the nonlinear term at high-frequency of $b$ equation.
All decay estimates in this paper are justified by using the continuity method. We will explain how to address the challenging aspects in different situations in three parts:

{\bf 1. $L^q$ decay.}

First, according to the difference between high and low frequencies, the nonlinear term is  divided into two parts,
\begin{align}\nonumber
	&\left\|\int_{0}^{t}K_1(t-\tau)\left(b\cdot\nabla u-u\cdot\nabla b\right)(\tau)\dd \tau\right\|_{L^q(\mr^2)}\\\label{b-nonlinear-lq}
    &\leq \left\|\int_{0}^{t}\w{K_1(t-\tau)\nabla\cdot\left(u\otimes b\right)}\dd \tau\right\|_{L^{q'}(S_1)}+\left\|\int_{0}^{t}\w{K_1(t-\tau)\nabla\cdot\left(u\otimes b\right)}\dd \tau\right\|_{L^{q'}(S_2)}
\end{align}
According to the above analysis, the most difficult term is the first term on the right side of \eqref{b-nonlinear-lq}. Using \eqref{fren-2} with $\vartheta=1$ in Lemma \ref{kernel prop} yields
\begin{align*}
	&\left\|\int_{0}^{t}\w{K_1(t-\tau)\nabla\cdot\left(u\otimes b\right)}\dd \tau\right\|_{L^{q'}(S_1)}\\
	&\leq C \int_{0}^t \lef \gamma^{-\frac{1}{2}}|\xi|^{-1-m}\e^{-\frac{1}{8\gamma}(t-\tau)} \rig_{L^{2}(\{\xi;4\gamma |\xi|^2\geq \frac{3}{4}\})} \left\|\Lambda^{m}\nabla\cdot\left(u(\tau)\otimes b(\tau)\right)\right\|_{L^{\frac{2q}{2+q}}(\mr^2)} \dd\tau\\
	&\leq C\gamma^{\frac{m-1}{2}}\int_{0}^{t}\e^{-\frac{1}{8\gamma}(t-\tau)}\left(\|\nabla\Lambda^{m}u(\tau)\|_{L^{2}}\|b(\tau)\|_{L^{q}}+\|\nabla\Lambda^{m}b(\tau)\|_{L^{2}}\|u(\tau)\|_{L^q}\right)\dd\tau.
\end{align*}
Substituting the assumption of the continuity argument
\begin{align*}
	\|u(t)\|_{L^q}+	\|b(t)\|_{L^q}\leq C_0t^{\frac{1}{q}-\frac{1}{2}}\left(1+\gamma\right)\left\| (u_0,b_0,a_0)\right\|_{X^m(\mr^2)},
\end{align*}
and then using a basic calculation referring to \eqref{p-3} in Lemma \ref{expintegral}
\begin{align*}
	\int_{0}^t\e^{-\frac{1}{4\gamma}(t-\tau)}\tau^{\frac{2}{q}-1}\dd\tau\leq C\gamma t^{\frac{2}{q}-1},
\end{align*}
it is deduced that
\begin{align*}
	&\left\|\int_{0}^{t}\w{K_1(t-\tau)\nabla\cdot\left(u\otimes b\right)}\dd \tau\right\|_{L^{q'}(S_1)}\\
	&\leq C\left(1+\gamma\right)\left(\gamma^{\frac{m}{2}}+\gamma^{\frac{m}{2}+\frac{m}{2m+2}}\right)\left(C_0\left\| (u_0,b_0,a_0)\right\|_{X^m(\mr^2)} \right)^2t^{\frac{1}{q}-\frac{1}{2}}.
\end{align*}

Regarding the decay estimate of high-order derivatives of $b$, since the nonlinear part of the equation for $b$ is coupled with $u$, it is more troublesome to  prove $H^\varrho (0\leq \varrho<m+1)$ decay  of $b$ directly. Thus, we divide it into two cases: $0\leq \varrho\leq m$ and $m<\varrho<m+1$.

{\bf 2. $H^\beta$ decay for $0\leq \beta \leq m$.}

By using \eqref{fren-2} with $\vartheta=1$ in Lemma \ref{kernel prop}, for the high frequency part of the nonlinear term, we have
\begin{align*}
&\left\|	\int_{0}^{t} \w{K_1(t-\tau)\Lambda^\beta\nabla\cdot\left(u\otimes b\right)}\dd \tau\right\|_{L^2(S_1)}\\
&\leq C \int_0^t \left\|\gamma^{-\frac{1}{2}}|\xi|^{-1}\e^{-\frac{1}{8\gamma}(t-\tau)}\w{\Lambda^\beta\nabla\cdot\left(u\otimes b\right)}(\xi)\right\|_{L^2(S_1)}\dd \tau\\
&\leq C\gamma^{-\frac{1}{2}}\int_0^t\e^{-\frac{1}{8\gamma}(t-\tau)} \left\|\Lambda^\beta(u(\tau)\otimes b(\tau))\right\|_{L^2(\mr^2)}\dd\tau.
\end{align*}
To offset the term $\gamma^{-\frac{1}{2}}$, we use \eqref{p-2} in Lemma \ref{expintegral}, which states that
\begin{align}\nonumber
	&\int_{0}^{t}\left(\e^{-\frac{1}{8\gamma}(t-\tau)}(1+\tau)^{\frac{1-\beta}{2}-\frac{1}{c}}\right)^\frac{2+2m}{1+2m}\dd \tau \\ \nonumber
	&\leq \begin{cases}
		C(1+t)^{-1}\left(\gamma +\gamma^2\right), &(\frac{1}{c}-\frac{1}{2}+\frac{\beta}{2})\frac{2+2m}{1+2m}=1,\\
		C(1+t)^{(\frac{1-\beta}{2}-\frac{1}{c})\frac{2+2m}{1+2m}}	\gamma , &(\frac{1}{c}-\frac{1}{2}+\frac{\beta}{2})\frac{2+2m}{1+2m}<1,\\
		C(1+t)^{(\frac{1-\beta}{2}-\frac{1}{c})\frac{2+2m}{1+2m}}\left(\gamma +\gamma^{(\frac{1}{c}-\frac{1}{2}+\frac{\beta}{2})\frac{2+2m}{1+2m}}\right),&(\frac{1}{c}-\frac{1}{2}+\frac{\beta}{2})\frac{2+2m}{1+2m}>1.
	\end{cases}
\end{align}
Therefore,
\begin{align*}\nonumber
&\left\|	\int_{0}^{t} \w{K_1(t-\tau)\Lambda^\beta\nabla\cdot\left(u\otimes b\right)}\dd \tau\right\|_{L^2(S_1)}\\
\leq& C\left(1+\gamma^{1+\frac{1}{c}-\frac{1-\beta}{2}}\right)\phi(\gamma)\left(C_0\|(u_0,b_0,a_0)\|_{X^{m,c}(\mr^2)}\right)^2(1+t)^{\frac{1-\beta}{2}-\frac{1}{c}}.
\end{align*}
It is noted that $\phi(\gamma)$ is a polynomial function with nonnegative exponents of $\gamma,$ which enables us to achieve the desired result.

{\bf 3. $H^\varrho$ decay for $m< \varrho <m+1$.}

We can deal with the nonlinear term in a similar way as above,
 when the higher-order  derivative is evaluated at $b.$ 
 However,  we should make full use of dissipative effect, when the higher-order  derivative is evaluated at $u.$  In fact,
  \begin{equation*}
 	\begin{split}
 		&\left\|\int_{0}^{t}K_1(t-\tau)\widehat{\Lambda^\varrho \left(b\cdot\nabla u-u\cdot\nabla b\right)}\dd \tau\right\|_{L^2(S_1)}\\
 		\leq& 	 C\gamma^{-\frac{1}{2}}\int_0^t\e^{-\frac{1}{8\gamma}(t-\tau)} \left\|\Lambda^\varrho(u(\tau)\otimes b(\tau))\right\|_{L^2}\dd\tau 
 	\end{split}
 \end{equation*}
Using the  interpolation inequality and the obtained $\beta$-order decay estimates, we  conclude	\begin{equation*}\begin{split}
		&\|\Lambda^\varrho u(\tau)\|_{L^{i_1}(\mathbb{R}^2)}\|b(\tau)\|_{L^{i_2}(\mathbb{R}^2)}\\ \leq&C\left\| \Lambda^{m+1}u(\tau)\right\|_{L^2}^{\theta_1}\left\| u(\tau)\right\|_{L^2}^{1-\theta_1} \|\Lambda^{1-\frac{2}{i_2}}b (\tau)\|_{L^2}
		\\ \leq&
		C\left(1+\gamma^{1+\frac{1}{c}-\frac{1}{i_2}}\right)\left(1+\gamma^{\frac{1}{2}+\frac{1}{c}}\right)^{1-\theta_1}\left(C_0\|(u_0,b_0,a_0)\|_{X^{m,c}(\mr^2)}  \right)^{2-\theta_1} \\
		&~~~~~~ ~~~~~~~~~~\times(1+\tau)^{\frac{1-\varrho}{2}-\frac{1}{c}} \left\| \Lambda^{m+1}u\right\|_{L^2}^{\theta_1},
\end{split}	\end{equation*}
where
\begin{align*}
	\begin{cases}
		\theta_1=\dfrac{\frac{1}{c}}{\frac{m}{2}+\frac{1}{c}}\\[2mm]
		\frac{1}{i_1}+\frac{1}{i_2}=\frac{1}{2}\\[2mm]
		\varrho-\frac{2}{i_1}=m\theta_1-1(1-\theta_1)
	\end{cases}\Rightarrow \frac{1}{i_2}-\frac{1}{c}+\left(\frac{1}{2}-\frac{1}{c}\right)\left(1-\theta_1\right)=\frac{1-\varrho}{2}-\frac{1}{c}.
\end{align*}
Then, after detailed calculation,
it yields
\begin{align*}\nonumber
	&C\gamma^{-\frac{1}{2}}\int_0^{t}\e^{-\frac{1}{8\gamma}(t-\tau)}\|\Lambda^\varrho u(\tau)\|_{L^{i_1}}\|b(\tau)\|_{L^{i_2}}\dd \tau\\ \leq & C\left(C_0\|(u_0,b_0,a_0)\|_{X^{m,c}(\R^2)}\right)^2\left(1+\gamma^{1+\frac{1}{c}-\frac{1-\varrho}{2}}\right)\phi_2(\gamma)(1+t)^{\frac{1-\varrho}{2}-\frac{1}{c}},
\end{align*}where $\phi_2(\gamma)$ is a polynomial function with nonnegative exponents of $\gamma.$

On the other hand, for the estimation  of  the low-frequency part, we also
need to concentrate more of the derivatives on $b$. More precisely,
 \begin{equation*}
	\begin{split}
		&\left\|\int_{0}^tK_1(t-\tau)\widehat{\Lambda^\varrho \left(b\cdot\nabla u-u\cdot\nabla b\right)\dd \tau}\right\|_{L^2(S_2)}\\ \leq&...+\int_{\frac{t}{2}}^t \left\|\e^{-|\xi|^2(t-\tau)}\w{\Lambda^\varrho \left(b\cdot\nabla u-u\cdot\nabla b\right)}\right\|_{L^2(S_2)}\dd\tau\\
		\leq &...+	C	\int_{\frac{t}{2}}^t (t-\tau)^{-\frac{\varrho-m}{2}-\frac{1}{k}} \left\|\Lambda^{m} \left(u\otimes b\right)\right\|_{L^k(\mr^2)}\dd\tau.
	\end{split}
\end{equation*}
By  choosing \begin{equation*}
	1-\frac{\varrho-m}{2}>\frac{1}{k}\geq\frac{1}{2}+\frac{1}{r_2}+\frac{m-\varrho}{2},~~\frac{1}{r_1}+\frac{1}{r_2}=\frac{1}{k},
\end{equation*}
\begin{equation*}\begin{split}
		r_2: \begin{cases}
			2< r_2<\min\{\frac{2}{1-m},\frac{2}{\varrho-m}\}, &\text{if} ~~m<1,\\[2mm]
			2< r_2<\frac{2}{\varrho-m},&\text{if} ~~m>1,
\end{cases}\end{split}	\end{equation*}
we can deduce
\begin{equation*}\begin{split}
	&\left\|\Lambda^{m} b(\tau)\right\|_{L^{r_1}}\|u(\tau)\|_{L^{r_2}}\\
	\leq& C\left\|\Lambda^{m+1-\frac{2}{r_1}}b(\tau)\right\|_{L^2}\|\Lambda^{1-\frac{2}{r_2}}u(\tau)\|_{L^{2}}\\
	\leq&  C\left(1+\gamma^{1+\frac{1}{c}-\frac{1-\varrho}{2}}\right)\left(1+\gamma^{1-\frac{1}{r_2}+\frac{1}{c}}\right)\left(C_0\|(u_0,b_0,a_0)\|_{X^{m,c}(\mr^2)}\right)^2(1+\tau)^{\frac{1-m}{2}-\frac{1}{c}-1+\frac{1}{k}}.
\end{split}	\end{equation*}
 Then, 
\begin{align*}
&	\int_{\frac{t}{2}}^t (t-\tau)^{-\frac{\varrho-m}{2}-\frac{1}{k}} \left\|\Lambda^{m} b(\tau)\right\|_{L^{r_1}}\|u(\tau)\|_{L^{r_2}}\dd\tau
\\	\leq  & C\left(1+\gamma^{1+\frac{1}{c}-\frac{1-\varrho}{2}}\right)\left(1+\gamma^{1-\frac{1}{r_2}+\frac{1}{c}}\right)\left(C_0\|(u_0,b_0,a_0)\|_{X^{m,c}(\mr^2)}\right)^2(1+\tau)^{\frac{1-m}{2}-\frac{1}{c}-1+\frac{1}{k}}.
\end{align*}

The rest of this paper is divided into three sections. In Section 2, we collect some auxiliary lemmas. Section 3 is devoted to proving Theorem \ref{thm1} and Section 4 is devoted to proving Theorem \ref{thm2}.

\section{Preliminaries}

In this section, we collect preliminaries. 
The following lemma is about commutator estimates and product inequalities, referring to \cite{kato-1988-CommutatorEstimatesEuler} and \cite{kenig}.

\begin{lemma}\label{commutator}
	Let $s>0,1<p<\infty$ and $\frac{1}{p}=\frac{1}{p_1}+\frac{1}{q_1}=\frac{1}{p_2}+\frac{1}{q_2}$.
	\begin{itemize}
		\item  If $f\in L^{p_1}(\R^{n})\cap W^{s,q_2}(\R^{n}),g\in L^{p_2}(\R^{n})\cap W^{s,q_1}(\R^{n})$, there is an absolute constant $C$ such that
		\begin{equation}\begin{split}\label{mul}
		&	\left\|\Lambda^s(f g)\right\|_{L^p(\R^{n})} \\ \leq& C\left(\|f\|_{L^{p_1}(\R^{n})}\left\|\Lambda^s g\right\|_{L^{q_1}(\R^{n})}+\|g\|_{L^{p_2}(\R^{n})}\left\|\Lambda^s f\right\|_{L^{q_2}(\R^{n})}\right).
	\end{split}
		\end{equation}
		\item For any $f\in W^{1,p_1}(\R^{n})\cap W^{s,q_2}(\R^{n}),g\in L^{p_2}(\R^{n})\cap W^{s-1,q_1}(\R^{n})$, there exists an absolute constant $C$ such that
		\begin{equation}\begin{split}\label{com}
				&	\left\|[\Lambda^s, f]g\right\|_{L^p(\R^{n})}\\ \leq& C\left(\|\nabla f\|_{L^{p_1}(\R^{n})}\left\|\Lambda^{s-1} g\right\|_{L^{q_1}(\R^{n})}+\|g\|_{L^{p_2}(\R^{n})}\left\|\Lambda^s f\right\|_{L^{q_2}(\R^{n})}\right),
		\end{split}\end{equation}
	where the bracket $[~,~]$ represents the commutator.
	\end{itemize} 
\end{lemma}

 The   Gagliardo-Nirenberg interpolation inequality and the Hausdorff-Young inequality as follows are useful in the proof of the main theorems.

\begin{lemma}[\cite{gag-nirenberg}]\label{gag}
Let $0\leq r< s_2<\infty$ and $1 < p_1, p_2, q \leq \infty.$ Then, there exists a positive constant $C$ such  that
	\begin{align*}
		\|\Lambda^rf\|_{L^q(\R^{n})}\leq C	\|\Lambda^{s_1}f\|^{\theta}_{L^{p_1}(\R^{n})}	\|\Lambda^{s_2}f\|^{1-\theta}_{L^{p_2}(\R^{n})},
	\end{align*}
with 
	\begin{align*}
		\frac{1}{q}-\frac{r}{n}=\theta\left(\frac{1}{p_1}-\frac{s_1}{n}\right)+(1-\theta)\left(\frac{1}{p_2}-\frac{s_2}{n}\right),
	\end{align*} where $0\leq\theta\leq1-\frac{r}{s_2}$ $(\theta\neq0 ~\text{if}~ q=\infty).$
\end{lemma}

\begin{lemma}[\cite{fourier-2009-nonlinear-ponce-hausdroffyoung}]\label{hausdorff-Young}
	 If $f\in L^p(\mathbb{R}^n), 1\leq p\leq 2$ and $\dfrac{1}{p}+\dfrac{1}{q}=1$, then $\hat{f}\in L^{q}(\mathbb{R}^n)$ which satisfies
	\begin{align*}
		\|\hat{f}\|_{L^{q}(\mathbb{R}^n)}\leq C \|f\|_{L^p(\mathbb{R}^n)}.
	\end{align*}
\end{lemma}

The next two lemmas provide the  integral representation and the properties of the kernel operator for system \eqref{mhdwave}. We can refer to \cite{wu-2022-wave} (Proposition 2.2 and Lemma 3.1) for proof.

\begin{lemma}\label{integral}
	System \eqref{mhdwave} can be equivalently expressed in terms of the following integral representation
\begin{align}\label{u integral}
	u(x,t)&=\e^{t\Delta}u_0+\int_0^t \e^{(t-\tau)\Delta}\mathbb{P}\left(b\cdot\nabla b-u\cdot\nabla u\right)(\tau)\dd \tau.\\\label{b integral}
	b(x,t)&=\left(K_0+\frac{1}{2}K_1\right)b_0+\gamma K_1a_0+\int_{0}^{t}K_1(t-\tau)\left(b\cdot\nabla u-u\cdot\nabla b\right)(\tau)\dd \tau.
\end{align}
Here $\mathbb{P}=I-\nabla \Delta^{-1}\diver$ represents the projection operator onto divergence-free vector fields, and the kernel functions $K_0$ and $K_1$ are defined as follows
\begin{align*}
	\w{K_0}(\xi,t)=\dfrac{1}{2}\left(\e^{\lambda_+t}+\e^{\lambda_-t}\right),~~ \w{K_1}(\xi,t)=\dfrac{\e^{\lambda_+t}-\e^{\lambda_-t}}{\gamma(\lambda_+-\lambda_-)}
\end{align*}
with
\begin{align*}
 \lambda_\pm=\dfrac{-1\pm \sqrt{1-4\gamma|\xi|^2}}{2\gamma}.
\end{align*}
\end{lemma}

\begin{lemma}\label{kernel prop}
		Define
	\begin{align*}
		S_1=\{\xi\in\mathbb{R}^2; 4\gamma|\xi|^2\geq \dfrac{3}{4}\},~~ S_2=\{\xi\in\mathbb{R}^2; 4\gamma|\xi|^2< \dfrac{3}{4}\}.
	\end{align*}
The definitions of $K_0$, $K_1$ and $\lambda_\pm$ are as shown in Lemma \ref{integral}. Then

	\noindent	(i) for $\xi\in S_1$, there holds
	\begin{align}\label{fren-1}
		|\w{K_0}(\xi,t)|, ~|\w{K_1}(\xi,t)|\leq C\e^{-\frac{1}{8\gamma}t}.
	\end{align}
More generally, for any $0\leq \vartheta\leq 1$,
	\begin{align}\label{fren-2}
	 |\w{K_1}(\xi,t)|\leq C\gamma^{-\frac{\vartheta}{2}}|\xi|^{-\vartheta}\e^{-\frac{1}{8\gamma}t}.
	\end{align}
	
	\noindent(ii) for $\xi\in S_2$, there holds
	\begin{align}\label{fren-3}
		|\w{K_0}(\xi,t)|, ~|\w{K_1}(\xi,t)|\leq C\left(\e^{-\frac{3}{4\gamma}t}+\e^{-|\xi|^2t}\right)\leq C\e^{-|\xi|^2t}.
	\end{align}
\end{lemma}
The lemma as follows is crucial in dealing with the decay estimates of the nonlinearity terms involving high frequency of $b$ equation.

\begin{lemma}\label{expintegral}
	Let $R>0$. Then for any $t>0$ and $\varkappa>0$,
	\begin{align}\label{p-1}
		\int_{0}^{t}\e^{-R(t-\tau)}(1+\tau)^{-\frac{1}{\varkappa}} \dd \tau\leq  \begin{cases}
			C(1+t)^{-1}\left(R^{-1}+R^{-2}\right), & \varkappa=1,\\
			C(1+t)^{-\frac{1}{\varkappa}}\left(1+R^{-1}\right),&\varkappa>1,\\
			C(1+t)^{-\frac{1}{\varkappa}}\left(1+R^{-\frac{1}{\varkappa}}+R^{-1}\right),&\varkappa<1.
		\end{cases}
	\end{align}
	and for any $t\geq 1$,
	\begin{align}\label{p-2}
		\int_{0}^{t}\e^{-R(t-\tau)}(1+\tau)^{-\frac{1}{\varkappa}} \dd \tau\leq  \begin{cases}
			C(1+t)^{-1}\left(R^{-1}+R^{-2}\right), & \varkappa=1,\\
			C(1+t)^{-\frac{1}{\varkappa}}R^{-1},&\varkappa>1,\\
			C(1+t)^{-\frac{1}{\varkappa}}\left(R^{-\frac{1}{\varkappa}}+R^{-1}\right),&\varkappa<1.
		\end{cases}
	\end{align}
	If $\varkappa>1$, then for any $t>0$,
	\begin{align}\label{p-3}
		\int_{0}^{t}\e^{-R(t-\tau)}\tau^{-\frac{1}{\varkappa}} \dd \tau\leq C t^{-\frac{1}{\varkappa}}R^{-1}.
	\end{align}
	Here, the constants $C's$ depend on $\varkappa$ only.
\end{lemma}

\begin{proof}
	To prove \eqref{p-1}, we divide $\int_{0}^{t}\e^{-R(t-\tau)}(1+\tau)^{-\frac{1}{\varkappa}} \dd \tau$ into two parts,
	\begin{equation}\label{p-p-1}\begin{split}
			&\int_{0}^{t}\e^{-R(t-\tau)}(1+\tau)^{-\frac{1}{\varkappa}} \dd \tau\\=&\int_{0}^{\frac{t}{2}}\e^{-R(t-\tau)}(1+\tau)^{-\frac{1}{\varkappa}} \dd \tau+\int_{\frac{t}{2}}^{t}\e^{-R(t-\tau)}(1+\tau)^{-\frac{1}{\varkappa}} \dd \tau.
		\end{split}
	\end{equation}
	The second part can be estimated as follows,
	\begin{align*}
		\int_{\frac{t}{2}}^{t}\e^{-R(t-\tau)}(1+\tau)^{-\frac{1}{\varkappa}} \dd \tau
		&\leq C(1+t)^{-\frac{1}{\varkappa}}\int_{\frac{t}{2}}^t \e^{-R(t-\tau)}\dd \tau \\ &\leq CR^{-1}(1+t)^{-\frac{1}{\varkappa}}.
	\end{align*}
	For the first part, we have
	\begin{align*}
		&\int_{0}^{\frac{t}{2}}\e^{-R(t-\tau)}(1+\tau)^{-\frac{1}{\varkappa}} \dd \tau\\
		\leq& \e^{-\frac{R}{2}t}\int_0^\frac{t}{2}(1+\tau)^{-\frac{1}{\varkappa}}\dd \tau\\ \leq& C\e^{-Rt}\times\begin{cases}
			\ln(1+t),&\varkappa=1,\\[2mm]
			(1+t)^{1-\frac{1}{\varkappa}},&\varkappa>1,\\[2mm]
			1,&\varkappa<1.
		\end{cases}\\ \leq&\begin{cases}
			C(1+t)^{-1}\left(R^{-1}+R^{-2}\right), & \varkappa=1,\\
			C(1+t)^{-\frac{1}{\varkappa}}\left(1+R^{-1}\right),&\varkappa>1,\\
			C(1+t)^{-\frac{1}{\varkappa}}\left(1+R^{-\frac{1}{\varkappa}}\right),&\varkappa<1.
		\end{cases}
	\end{align*}
	In the last inequality above, we have used the fact that, for any $m>0,$
	$$\sup_{t>0}\e^{-Rt}t^m\leq CR^{-m},$$ where $C$ is independent of $m.$
	Hence, putting these estimates together yields
	\begin{align*}
		\int_{0}^{t}\e^{-R(t-\tau)}(1+\tau)^{-\frac{1}{\varkappa}} \dd \tau\leq  \begin{cases}
			C(1+t)^{-1}\left(R^{-1}+R^{-2}\right), & \varkappa=1,\\[2mm]
			C(1+t)^{-\frac{1}{\varkappa}}\left(1+R^{-1}\right),&\varkappa>1,\\[2mm]
			C(1+t)^{-\frac{1}{\varkappa}}\left(1+R^{-\frac{1}{\varkappa}}+R^{-1}\right),&\varkappa<1,
		\end{cases}
	\end{align*} which is \eqref{p-1}.

	If  $t\geq 1$, the second term on the right of \eqref{p-p-1} can be  estimated as before. However, using the fact that for any $m>0,$
	$$\sup_{t\geq1}\e^{-Rt}(1+t)^m\leq CR^{-m},~~\text{$C$ being independent of $m$},$$ 
	the first term on the right of \eqref{p-p-1} can be replaced by the following
	\begin{align*}
		&\int_{0}^{\frac{t}{2}}\e^{-R(t-\tau)}(1+\tau)^{-\frac{1}{\varkappa}} \dd \tau \\ \leq&\begin{cases}
			C(1+t)^{-1}\left(R^{-1}+R^{-2}\right), & \varkappa=1,\\
			C(1+t)^{-\frac{1}{\varkappa}}R^{-1},&\varkappa>1,\\
			C(1+t)^{-\frac{1}{\varkappa}}R^{-\frac{1}{\varkappa}},&\varkappa<1.
		\end{cases}
	\end{align*}
	At the point, \eqref{p-2} can also be proved.
	
	The proof of \eqref{p-3} is similar to that of \eqref{p-1}. However, in order to guarantee the integrability of $\int_{0}^\frac{t}{2}\tau^{-\frac{1}{\varkappa}}\dd\tau $, an additional condition of $\varkappa>1$ needs to be imposed.
\end{proof}

At the end of this section, we give two lemmas about heat kernel estimates and  estimates for the operator  $K(D)$, respectively.

\begin{lemma}[\cite{miao-2008-heat}]\label{heatkernel}
Let $s\geq 0,~1\leq p\leq q \leq \infty.$ Then, there holds for any $f\in L^p(\mathbb{R}^n),$
	\begin{align*}
		\left\|\Lambda^s e^{\Delta t } f\right\|_{L^q\left(\mathbb{R}^n\right)} \leq C t^{-\frac{s}{2}-\frac{n}{2 }\left(\frac{1}{p}-\frac{1}{q}\right)}\|f\|_{L^p\left(\mathbb{R}^n\right)}.
	\end{align*}
\end{lemma}

\begin{lemma}\label{young+KDf}
	Let $1\leq r\leq 2\leq q\leq \infty, p\geq q'=\frac{q}{q-1}$ and $\frac{1}{p}+\frac{1}{q}=\frac{1}{r}$. Suppose $ \Omega\subset \mr^2$ is any subspace in the frequency space, then the following holds
		\begin{align*}
			\left\|\w{K(D)f}\right\|_{L^{q'}(\Omega)}\leq C\left\|\w{K}(\xi)\right\|_{L^{p}(\Omega)}\left\|f(x)\right\|_{L^{r}(\mr^2)}.
	\end{align*}
	\begin{proof}
		Since
		\begin{align*}
		1-\dfrac{1}{q}=\dfrac{1}{p}+1-\dfrac{1}{r} \Rightarrow \dfrac{1}{p}+\dfrac{1}{q}=\dfrac{1}{r}.
	\end{align*}
	Using  Lemma \ref{hausdorff-Young} and H\"{o}lder's inequality, we have
	\begin{equation*}\begin{split}
		\left\|\w{K(D)f}\right\|_{L^{q'}(\Omega)}&=\left\|\w{K}(\xi)\w{f}(\xi)\right\|_{L^{q'}(\Omega)}\\
		&\leq C\left\|\w{K}(\xi)\right\|_{L^{p}(\Omega)}\left\|\w{f}(\xi)\right\|_{L^{r'}(\Omega)}\\
		&\leq C\left\|\w{K}(\xi)\right\|_{L^{p}(\Omega)}\left\|\w{f}(\xi)\right\|_{L^{r'}(\mr^2)}\\
		&\leq C\left\|\w{K}(\xi)\right\|_{L^{p}(\Omega)}\left\|f(x)\right\|_{L^{r}(\mr^2)}.\end{split}
	\end{equation*}
	\end{proof}
\end{lemma}

\section{Global well-posedness and $L^q~(q\geq 2)$ decay}

In this section, we are going to prove Theorem \ref{thm1}. The proof is divided into two subsections, we  will establish the global existence and uniqueness  of  solutions to \eqref{mhdwave}  in the first subsection. In the second subsection, we will prove the $L^q$ decay estimates for $q > 2$.

\subsection{Global existence and uniqueness }

The goal of this subsection is to prove global existence and uniqueness  of  solutions to \eqref{mhdwave}.
First of all, we derive the a priori estimates for solutions of \eqref{mhdwave} as follows.

\begin{lemma}\label{lem1}
 Let $m>0$ and $\gamma>0$. Suppose that $(u, b)$ is a solution of system \eqref{mhdwave},
then there exists a  constant $\varepsilon>0$ such that, if  
\begin{align*}
	\left\|(u_0,b_0,a_0)\right\|_{X^m(\mr^2)}\leq \varepsilon,
\end{align*}
then, the solution $(u,b)$ satisfies, for any $t>0$,
\begin{align}\nonumber
	&\|u(t)\|_{H^m}^2+\| b(t)\|_{H^m}^2+2\gamma^2  \|\partial_t  b(t)\|_{H^m}^2+2\gamma\|\nabla b(t)\|_{H^m}^2\\\nonumber
	&~~~~~~~+\int_0^t \|\nabla u(\tau)\|_{H^m}^2+\|\nabla b(\tau)\|_{H^m}^2+\gamma\|\partial_\tau b(\tau)\|_{H^m}^2 \dd\tau \\\label{pri}
	\leq& C\left(1+\gamma^{\frac{2m}{2m+2}}\right)\left\| (u_0,b_0,a_0)\right\|_{X^m(\mr^2)}^2.
\end{align}

\end{lemma}
\begin{proof}
	Applying the operator $\Lambda^m ~(m>0)$ to both sides of equations $\eqref{mhdwave}_1$ and $\eqref{mhdwave}_2$, taking the inner product of the resulting equations with $\Lambda^m u$ and $\Lambda^m b$ respectively, and summing them up give to
\begin{equation}\label{3.1}
	\begin{split}
	&\dfrac{1}{2}\dfrac{\dd}{\dd t}\left(\|\Lambda^m u\|_{L^2}^2+\|\Lambda^m b\|_{L^2}^2\right)+\|\Lambda^{m+1}u\|_{L^2}^2+\|\Lambda^{m+1}b\|_{L^2}^2+\gamma\int_{\mathbb{R}^2}\Lambda^m b\cdot \partial_{tt} \Lambda^m b\dd x\\
	=&-\int_{\mathbb{R}^2}\Lambda^m\left(u\cdot\nabla u\right)\cdot\Lambda^m u\dd x+\int_{\mathbb{R}^2}\Lambda^m\left(b\cdot\nabla b\right)\cdot\Lambda^m u\dd x\\
	&~~~~+\int_{\mathbb{R}^2}\Lambda^m\left(b\cdot\nabla u\right)\cdot\Lambda^m b\dd x-\int_{\mathbb{R}^2}\Lambda^m\left(u\cdot\nabla b\right)\cdot\Lambda^m b\dd x\\
	\stackrel{\mathrm{def}}{=}& I_1+I_2+I_3+I_4.
\end{split}
\end{equation}
In order to deal with the bad term $\gamma\int_{\mathbb{R}^2}\Lambda^m b\cdot \partial_{tt} \Lambda^m b\dd x,$ 
we first get that by some basic calculations,
\begin{align*}
	\int_{\mathbb{R}^2}\Lambda^m b\cdot \partial_{tt} \Lambda^m b\dd x=\partial_t\int_{\mathbb{R}^2}\Lambda^m b\cdot \partial_t\Lambda^m b\dd x-\int_{\mathbb{R}^2}\left|\partial_t \Lambda^m b\right|^2\dd x.
\end{align*}
Then, 	applying the operator $\Lambda^m$ to $\eqref{mhdwave}_2$ and
multiplying the resulting equation  by $\partial_t\Lambda^m b$ yield
\begin{equation}\label{3.2}
	\begin{split}
	&\dfrac{1}{2}\dfrac{\dd}{\dd t}\left(\gamma \|\partial_t \Lambda^m b\|_{L^2}^2\right)+\|\partial_t\Lambda^m b\|_{L^2}^2+\dfrac{1}{2}\dfrac{\dd}{\dd t}\|\Lambda^{m+1}b\|_{L^2}^2\\
	=&\int_{\mathbb{R}^2}\Lambda^m(b\cdot\nabla u)\cdot\partial_t \Lambda^m b\dd x-\int_{\mathbb{R}^2}\Lambda^m(u\cdot\nabla b)\cdot\partial_t \Lambda^m b\dd x\\
	\stackrel{\mathrm{def}}{=}& J_1+J_2.
\end{split}
\end{equation}
By multiplying \eqref{3.2} by $2\gamma$ and adding it to \eqref{3.1}, we obtain
\begin{equation}\begin{split}\label{3.3}
&	\dfrac{1}{2}\dfrac{\dd}{\dd t}\left(\|\Lambda^m u\|_{L^2}^2+\|\Lambda^m b\|_{L^2}^2+2\gamma^2  \|\partial_t \Lambda^m b\|_{L^2}^2+2\gamma\|\Lambda^{m+1}b\|_{L^2}^2+2\gamma\int_{\mathbb{R}^2}\partial_t\Lambda^m b\cdot\Lambda^m b\dd x\right)\\
	&~~~~~~~~~~~~~~~~~~~~~+\|\Lambda^{m+1}u\|_{L^2}^2+\|\Lambda^{m+1}b\|_{L^2}^2+\gamma\|\partial_t \Lambda^mb\|_{L^2}^2\\=&2\gamma J_1+2\gamma J_2+I_1+I_2+I_3+I_4.\end{split}
\end{equation}
The terms $J_1$ and $J_2$ can be bounded  as follows, by  H\"{o}lder's inequality and \eqref{mul},
\begin{equation*}\begin{split}
	2\gamma J_1+2\gamma J_2&\leq 2\gamma \|\partial_t \Lambda^m b\|_{L^2}\|\Lambda^{m+1}(b\otimes u)\|_{L^2}\\
	&\leq 2\gamma \|\partial_t \Lambda^m b\|_{L^2}\left(\|\Lambda^{m+1} u\|_{L^2}\|b\|_{L^\infty}+\|\Lambda^{m+1} b\|_{L^2}\|u\|_{L^\infty}\right).
\end{split}
\end{equation*}
Similarly,
\begin{align*}
	&	I_1+I_2+I_3+I_4\\ \leq& C\left(\|\Lambda^m u\|_{L^2}+\|\Lambda^m b\|_{L^2}\right)\left(\|\Lambda^{m+1}u\|_{L^2}+\|\Lambda^{m+1} b\|_{L^2}\right)\left(\|u\|_{L^\infty}+\| b\|_{L^\infty}\right).
\end{align*}
By utilizing Lemma \ref{gag}, we  derive
\begin{align}\label{3.4}
	\|u\|_{L^\infty}\leq C\|u\|_{L^2}^\frac{m}{m+1}\|\Lambda^{m+1} u\|_{L^2}^\frac{1}{m+1},~~
	\|\Lambda^m u\|_{L^2}\leq C\|\Lambda^{m+1}u\|_{L^2}^\frac{m}{m+1}\|u\|_{L^2}^\frac{1}{m+1}.
\end{align}
For the sake of simplicity in the following discussion, we denote
\begin{align*}
	X_m(t)&=\|\Lambda^m u\|_{L^2}^2+\|\Lambda^m b\|_{L^2}^2+2\gamma^2  \|\partial_t \Lambda^m b\|_{L^2}^2+2\gamma\|\Lambda^{m+1}b\|_{L^2}^2,\\
	Y_m(t)&=2\gamma\int_{\mathbb{R}^2}\partial_t\Lambda^m b\cdot\Lambda^m b\dd x,\\
	Z_m(t)&=\|\Lambda^{m+1}u\|_{L^2}^2+\|\Lambda^{m+1}b\|_{L^2}^2+\gamma\|\partial_t \Lambda^mb\|_{L^2}^2.
\end{align*}
Therefore, using \eqref{3.4} and Young's inequality, it follows that
\begin{align*}
	I_1+I_2+I_3+I_4&\leq C\left(\|\Lambda^{m+1}u\|_{L^2}^2+\|\Lambda^{m+1} b\|_{L^2}^2\right)\left(\|u\|_{L^2}+\| b\|_{L^2}\right)\\
	&\leq CZ_m(t)\left(\|u\|_{L^2}+\| b\|_{L^2}\right),
\end{align*}
and
\begin{align*}
	&2\gamma\left(J_1+J_2\right)\\
	\leq &C\gamma \|\partial_t\Lambda^m b\|_{L^2}\left(\|\Lambda^{m+1}u\|_{L^2}+\|\Lambda^{m+1} b\|_{L^2}\right)\\
	&~~~~~~~~~~\left(\|u\|_{L^2}^\frac{m}{m+1}\|\Lambda^{m+1} u\|_{L^2}^\frac{1}{m+1}+\|b\|_{L^2}^\frac{m}{m+1}\|\Lambda^{m+1} b\|_{L^2}^\frac{1}{m+1}\right)\\
	\leq &C\gamma^{\frac{m+2}{2m+2}} \gamma^{\frac{m}{2m+2}}  \|\partial_t\Lambda^m b\|_{L^2}^\frac{1}{m+1}\|\partial_t\Lambda^m b\|_{L^2}^\frac{m}{m+1}\\
	&~~~~~~~~~~\left(\|u\|_{L^2}^\frac{m}{m+1}\|\Lambda^{m+1} u\|_{L^2}^\frac{1}{m+1}+\|b\|_{L^2}^\frac{m}{m+1}\|\Lambda^{m+1} b\|_{L^2}^\frac{1}{m+1}\right)\left(\|\Lambda^{m+1}u\|_{L^2}+\|\Lambda^{m+1} b\|_{L^2}\right)\\
	\leq &C\gamma^{\frac{m+2}{2m+2}}  Z_m(t)\|\partial_t\Lambda^m b\|_{L^2}^\frac{1}{m+1}\left(\|u\|_{L^2}^\frac{m}{m+1}+\|b\|_{L^2}^\frac{m}{m+1}\right).
\end{align*}
Substituting the above estimates into \eqref{3.3} and integrating both sides of the resulting inequality from 0 to $t$ yield
\begin{align}\nonumber
	&X_m(t)+Y_m(t)+2\int_{0}^{t}Z_m(\tau)\dd \tau\\\nonumber
	\leq &	X_m(0)+Y_m(0)+C\int_{0}^{t}Z_m(\tau)\times\\\label{3.5}
	&~~~~\left[\left(\|u(\tau)\|_{L^2}+\| b(\tau)\|_{L^2}\right)+\gamma^{\frac{m+2}{2m+2}}\|\partial_\tau\Lambda^m b(\tau)\|_{L^2}^\frac{1}{m+1}\left(\|u(\tau)\|_{L^2}^\frac{m}{m+1}+\|b(\tau)\|_{L^2}^\frac{m}{m+1}\right)\right]\dd \tau.
\end{align}
Since  utilizing  H\"{o}lder's inequality and  Young's inequality gives to
\begin{align*}
	|Y_m(t)|\leq	2\gamma\int_{\mathbb{R}^2}|\partial_t\Lambda^m b\cdot\Lambda^m b|\dd x \leq \dfrac{2\|\Lambda^m b\|_{L^2}^2}{3}+\dfrac{3\gamma^2\|\partial_t \Lambda^m b\|_{L^2}^2}{2},
\end{align*}
and 
\begin{equation*}\begin{split}
	Y_m(0)
	&\leq 2\gamma\|\partial_t\Lambda^m b_0\|_{L^2}\|\Lambda^m b_0\|_{L^2}\\
	&\leq \sqrt{2}\gamma^2\|\partial_t\Lambda^m b_0\|_{L^2}^2+\dfrac{1}{\sqrt{2}}\|\Lambda^m b_0\|_{L^2}^2\\
	&\leq \dfrac{1}{\sqrt{2}}X_m(0)\leq X_m(0),
\end{split}\end{equation*}
we can derive 
\begin{equation*}\begin{split}
	&X_m(t)+Y_m(t)\\
	\geq &\|\Lambda^m u\|_{L^2}^2+\|\Lambda^m b\|_{L^2}^2+2\gamma^2  \|\partial_t \Lambda^m b\|_{L^2}^2+2\gamma\|\Lambda^{m+1}b\|_{L^2}^2\\&~~~~~~~~~~~~~~~~~~~~~~~~~~~~~~~~~~~~~~~-\left(\dfrac{2\|\Lambda^m b\|_{L^2}^2}{3}+\dfrac{3\gamma^2\|\partial_t \Lambda^m b\|_{L^2}^2}{2}\right)\\
	\geq &\dfrac{1}{3}X_m(t).\end{split}
\end{equation*}
This together with  \eqref{3.5} yields
\begin{align*}
	&X_m(t)+6\int_{0}^{t}Z_m(\tau)\dd \tau\\
	\leq &6X_m(0)+C\int_{0}^{t}Z_m(\tau)\times\\
	&~~~~\left[\left(\|u(\tau)\|_{L^2}+\| b(\tau)\|_{L^2}\right)+\gamma^{\frac{m+2}{2m+2}}\|\partial_\tau\Lambda^m b(\tau)\|_{L^2}^\frac{1}{m+1}\left(\|u(\tau)\|_{L^2}^\frac{m}{m+1}+\|b(\tau)\|_{L^2}^\frac{m}{m+1}\right)\right]\dd \tau.
\end{align*}
It is easy to verify that the above inequality  also holds when  $m=0$. 
Hence, if we  introduce two new notations
\begin{align*}
	X(t)&=\|u\|_{H^m}^2+\| b\|_{H^m}^2+2\gamma^2  \|\partial_t  b\|_{H^m}^2+2\gamma\|\nabla b\|_{H^m}^2,\\[2mm]
	Z(t)&=\|\nabla u\|_{H^m}^2+\|\nabla b\|_{H^m}^2+\gamma\|\partial_t b\|_{H^m}^2,
\end{align*}
then, we obtain 
\begin{equation}\label{e-1}
	\begin{split}
		&X(t)+6\int_{0}^{t}Z(\tau)\dd \tau\\
		\leq& CX(0)+C(1+\gamma^{\frac{m}{2m+2}})\int_{0}^{t}Z(\tau)X^\frac{1}{2}(\tau)\dd\tau.
\end{split}\end{equation}
Set
\begin{align*}
	E(t)=\sup_{\tau\in[0,t]}\left(X(\tau)+6\int_{0}^\tau Z(s)\dd \tau\right),
\end{align*}
\eqref{e-1} implies that
\begin{align*}
	E(t)\leq CE(0)+C(1+\gamma^{\frac{m}{2m+2}})E(t)^{\frac{3}{2}}.
\end{align*}
By employing the bootstrapping argument, we can find a small positive number $\varepsilon$ satisfying  \eqref{1.1}, i.e.,
\begin{align*}
	\left\| (u_0,b_0,a_0)\right\|_{X^m(\mr^2)}=	\|u_0\|_{H^m}+\|b_0\|_{H^{m+1}}+\|a_0\|_{H^m}\leq \varepsilon.
\end{align*}
Here \begin{align*}
	\varepsilon<\dfrac{1}{C\left(1+\gamma^{\frac{m}{2m+2}}\right)^{2}},
\end{align*} and the constant $C$ is independent of $\gamma$.
Then it holds
\begin{align*}
	E(t)\leq C\left(1+\gamma^{\frac{2m}{2m+2}}\right)\left\| (u_0,b_0,a_0)\right\|_{X^m(\mr^2)}^2.
\end{align*}
At this point, we have obtained \eqref{pri}.

\end{proof}

\begin{proof}[Proof of the Theorem \ref{thm1}]
	
{\bf  (global existence)} 

Via the analysis in \cite{wave-2021-jde-fouriersobolev}, we are able to get the  local existence result. Then, the global existence of solution to system \eqref{mhdwave} follows from the local existence and
the a priori estimates  established in  Lemma \ref{lem1} via standard continuity argument.  We omit the details here for brevity.

{\bf  (uniqueness)}

Let $(u^1,b^1)$ and $(u^2,b^2)$ be two solutions of system  \eqref{mhdwave} in the regularity class of \eqref{th1-1}. It's easy to compute  that  their difference $(\delta u,\delta b)$ and $\delta p$
\begin{align*}
	\delta u=u^1-u^2, ~\delta b=b^1-b^2, ~\delta p=p^1-p^2
\end{align*}
satisfying
\begin{equation}
	\begin{split}\label{u-1}
	\begin{cases}
			\partial_t \delta u-\Delta \delta u+u^1 \cdot \nabla \delta u+\delta u\cdot\nabla u^2+\nabla \delta p=b^1\cdot \nabla \delta b+ \delta b\cdot\nabla b^2,~ x\in\mathbb{R}^2,t>0, \\[1mm]
		\gamma \partial_{tt}\delta b+	\partial_t \delta b-\Delta \delta b+u^1\cdot \nabla \delta b+\delta u\cdot\nabla b^2=b^1\cdot\nabla \delta u+\delta b\cdot\nabla u^2, \\[1mm]
		\nabla\cdot \delta u=\nabla \cdot \delta b=0, \\[1mm]
		\delta u(x, 0)=0,\,\, (\partial_t \delta b)(x,0)=\delta b(x, 0)=0.
	\end{cases}\end{split}
\end{equation}
By taking the $L^2$-inner product to  \eqref{u-1}$_1$ and \eqref{u-1}$_2$ with  $\delta u$ and $\delta b$, respectively, we get
	\begin{align}\nonumber
	&\dfrac{1}{2}\dfrac{\dd}{\dd t}\left(\|  \delta u\|_{L^2}^2+\|  \delta b\|_{L^2}^2\right)+\|\nabla\delta u\|_{L^2}^2+\|\nabla\delta b\|_{L^2}^2+\gamma\int_{\mathbb{R}^2}  \delta b\cdot \partial_{tt}   \delta b\dd x\\\nonumber
	=&-\int_{\mathbb{R}^2} \left(\delta u\cdot\nabla u^2\right)\cdot  \delta u\dd x+\int_{\mathbb{R}^2} \left(\delta b\cdot\nabla b^2\right)\cdot  \delta u\dd x\\\label{u.2}
	&~~~~+\int_{\mathbb{R}^2} \left(\delta b\cdot\nabla u^2\right)\cdot  \delta b\dd x-\int_{\mathbb{R}^2} \left(\delta u\cdot\nabla b^2\right)\cdot  \delta b\dd x\\ \stackrel{\mathrm{def}}{=}& \sum_{i=1}^{4} I_i.
\end{align}
By computing the $L^2$-inner product of $\eqref{u-1}_2$ with respect to the time derivative term $\partial_t \delta b,$ we have
\begin{equation}\label{u.3}\begin{split}
&\dfrac{1}{2}\dfrac{\dd}{\dd t}\left(\gamma \|\partial_t   \delta b\|_{L^2}^2\right)+\|\partial_t  \delta b\|_{L^2}^2+\dfrac{1}{2}\dfrac{\dd}{\dd t}\|\nabla \delta b\|_{L^2}^2\\
	=&\int_{\mathbb{R}^2} (b^1\cdot\nabla \delta u)\cdot\partial_t   \delta b\dd x-\int_{\mathbb{R}^2} (u^1\cdot\nabla \delta b)\cdot\partial_t   \delta b\dd x\\
	&~~~~+\int_{\mathbb{R}^2} (\delta b\cdot\nabla u^2)\cdot\partial_t   \delta b\dd x-\int_{\mathbb{R}^2} (\delta u\cdot\nabla b^2)\cdot\partial_t   \delta b\dd x\\
	\stackrel{\mathrm{def}}{=}& \sum_{j=1}^{4}J_j.
\end{split}\end{equation}
In a similar fashion to get \eqref{3.3}, we have
\begin{equation}\begin{split}\label{3.8}
&	\dfrac{1}{2}\dfrac{\dd}{\dd t}\left(\|  \delta u\|_{L^2}^2+\|  \delta b\|_{L^2}^2+2\gamma^2  \|\partial_t   \delta b\|_{L^2}^2+2\gamma\| \nabla\delta b\|_{L^2}^2+2\gamma\int_{\mathbb{R}^2}\partial_t  \delta b\cdot  \delta b\dd x\right)\\
	&~~~~~~+\| \nabla\delta u\|_{L^2}^2+\| \nabla\delta b\|_{L^2}^2+\gamma\|\partial_t  \delta b\|_{L^2}^2\\=&\sum_{i=1}^{4} I_i+2\gamma \sum_{j=1}^{4}J_j.\end{split}
\end{equation}
Using H\"{o}lder's inequality and Lemma \ref{gag}, we can obtain 
\begin{align*}
	I_1&\leq \left|\int_{\mathbb{R}^2} \left(\delta u\cdot\nabla u^2\right)\cdot  \delta u\dd x\right|\\
	&\leq \left\| \nabla u^2\right\|_{L^2}\left\| \delta u\right\|_{L^4}^2\\ &\leq C\left\| \nabla u^2\right\|_{L^2}\left\| \delta u\right\|_{L^2}\left\|\nabla \delta u\right\|_{L^2}\\
	&\leq \frac{1}{16}\left\|\nabla \delta u\right\|_{L^2}^2+C\left\| \nabla u^2\right\|_{L^2}^2\left\| \delta u\right\|_{L^2}^2.
\end{align*}
In a similar way, we obtain
\begin{align*}
	 \sum_{i=2}^{4} I_i\leq \dfrac{3}{16}\left(\left\|\nabla \delta u\right\|_{L^2}^2+\left\|\nabla \delta b\right\|_{L^2}^2\right)+C\left(\left\| \nabla u^2\right\|_{L^2}^2+\left\| \nabla b^2\right\|_{L^2}^2\right)\left(\left\| \delta u\right\|_{L^2}^2+\left\| \delta b\right\|_{L^2}^2\right),
\end{align*}
and
\begin{align*}
	\left|J_1\right|+\left|J_2\right|&\leq \left|\int_{\mathbb{R}^2} (b^1\cdot\nabla \delta u)\cdot\partial_t   \delta b\dd x\right|+\left|\int_{\mathbb{R}^2} (u^1\cdot\nabla \delta b)\cdot\partial_t   \delta b\dd x\right|\\
	&\leq C \left\| \partial_t \delta b\right\|_{L^2}\left(\left\| b^1\right\|_{L^\infty}\left\|\nabla \delta u\right\|_{L^2}+\left\|u^1\right\|_{L^\infty}\left\|\nabla\delta b\right\|_{L^2}\right)\\
	&\leq \dfrac{1}{16\gamma}\left(\left\|\nabla\delta b\right\|_{L^2}^2+\left\|\nabla\delta u\right\|_{L^2}^2\right)+C\gamma\left\| \partial_t \delta b\right\|_{L^2}^2\left(\left\| b^1\right\|_{L^\infty}^2+\left\| u^1\right\|_{L^\infty}^2\right).
\end{align*}
On one hand, for the case when $0<m<1$, using  H\"{o}lder's inequality, Young's inequality and  the Sobolev embedding yields that 
\begin{equation*}\begin{split}
		\left|J_3\right|+\left|J_4\right|&\leq \left|\int_{\mathbb{R}^2} (\delta b\cdot\nabla u^2)\cdot\partial_t   \delta b\dd x\right|+\left|\int_{\mathbb{R}^2} (\delta u\cdot\nabla b^2)\cdot\partial_t   \delta b\dd x\right|\\
		&\leq C \left\| \partial_t \delta b\right\|_{L^2}\left(\left\| \nabla b^2\right\|_{L^\frac{2}{1-m}}\left\| \delta u\right\|_{L^\frac{2}{m}}+\left\|\nabla u^2\right\|_{L^\frac{2}{1-m}}\left\|\delta b\right\|_{L^\frac{2}{m}}\right)\\
		&\leq C \left\| \partial_t \delta b\right\|_{L^2}\left(\left\| \nabla b^2\right\|_{H^m}\left\| \delta u\right\|_{L^2}^m\left\| \nabla\delta u\right\|_{L^2}^{1-m}+\left\|\nabla u^2\right\|_{H^m}\left\|\delta b\right\|_{L^2}^m\left\| \nabla\delta b\right\|_{L^2}^{1-m}\right)\\
		&\leq \dfrac{1}{16\gamma^{1-m}} \left\| \partial_t \delta b\right\|_{L^2}^{2m}\left(\left\| \nabla\delta b\right\|_{L^2}^{2-2m}+\left\| \nabla\delta u\right\|_{L^2}^{2-2m}\right)\\
		&~~~~~~+C\gamma^{1-m}\left\| \partial_t \delta b\right\|_{L^2}^{2-2m}\left(\left\| \nabla b^2\right\|_{H^m}^2\left\| \delta u\right\|_{L^2}^{2m}+\left\|\nabla u^2\right\|_{H^m}^2\left\|\delta b\right\|_{L^2}^{2m}\right)\\
		&\leq \dfrac{1}{16\gamma} \left(\gamma \left\| \partial_t \delta b\right\|_{L^2}^{2}+\left\| \nabla\delta b\right\|_{L^2}^{2}+\left\| \nabla\delta u\right\|_{L^2}^{2}\right)\\
			&~~~~~~+C\frac{1}{\gamma^{1-m}}\left(\left\| \nabla b^2\right\|_{H^m}^2+\left\|\nabla u^2\right\|_{H^m}^2\right)\left(2\gamma^2\left\| \partial_t \delta b\right\|_{L^2}^{2}+\left\| \delta b\right\|_{L^2}^{2}+\left\|\delta u\right\|_{L^2}^{2}\right).\end{split}
\end{equation*}
On the other hand, for the case when $m\geq1$, it follows that by using  H\"{o}lder's inequality, Young's inequality and  the Sobolev embedding again
\begin{align*}
	\left|J_3\right|+\left|J_4\right|&\leq \left|\int_{\mathbb{R}^2} (\delta b\cdot\nabla u^2)\cdot\partial_t   \delta b\dd x\right|+\left|\int_{\mathbb{R}^2} (\delta u\cdot\nabla b^2)\cdot\partial_t   \delta b\dd x\right|\\
	&\leq C \left\| \partial_t \delta b\right\|_{L^2}\left(\left\| \nabla b^2\right\|_{L^4}\left\| \delta u\right\|_{L^4}+\left\|\nabla u^2\right\|_{L^4}\left\|\delta b\right\|_{L^4}\right)\\
	&\leq C \left\| \partial_t \delta b\right\|_{L^2}\left(\left\| \nabla b^2\right\|_{H^m}\left\| \delta u\right\|_{L^2}^\frac{1}{2}\left\| \nabla\delta u\right\|_{L^2}^{\frac{1}{2}}+\left\|\nabla u^2\right\|_{H^m}\left\|\delta b\right\|_{L^2}^\frac{1}{2}\left\| \nabla\delta b\right\|_{L^2}^\frac{1}{2}\right)\\
	&\leq \dfrac{1}{16\gamma} \left(\gamma \left\| \partial_t \delta b\right\|_{L^2}^{2}+\left\| \nabla\delta b\right\|_{L^2}^{2}+\left\| \nabla\delta u\right\|_{L^2}^{2}\right)\\
	&~~~~~~+C\gamma^{-\frac{1}{2}}\left(\left\| \nabla b^2\right\|_{H^m}^2+\left\|\nabla u^2\right\|_{H^m}^2\right)\left(2\gamma^2\left\| \partial_t \delta b\right\|_{L^2}^{2}+\left\| \delta b\right\|_{L^2}^{2}+\left\|\delta u\right\|_{L^2}^{2}\right).
\end{align*}
Therefore, summarizing the above estimates, one can get
\begin{align*}
 \sum_{j=1}^{4}J_j&\leq \dfrac{1}{8\gamma}\left(\gamma\left\| \partial_t \delta b\right\|_{L^2}^{2}+\left\| \nabla\delta b\right\|_{L^2}^{2}+\left\| \nabla\delta u\right\|_{L^2}^{2}\right)+C\left(\gamma^{-1}+\gamma^{-1+m}+\gamma^{-\frac{1}{2}}\right)\\
 &~~~~~\left(\left\|(b^1,b^2)\right\|_{H^{m+1}}^2+\left\|(u^1, u^2)\right\|_{H^{m+1}}^2\right)\left(2\gamma^2\left\| \partial_t \delta b\right\|_{L^2}^{2}+\left\| \delta b\right\|_{L^2}^{2}+\left\|\delta u\right\|_{L^2}^{2}\right).
\end{align*}
By plugging  the above estimates into \eqref{3.8} and integrating with respect to $t$, we obtain
\begin{align*}
&\|\delta u\|_{L^2}^2+\|  \delta b\|_{L^2}^2+2\gamma^2  \|\partial_t   \delta b\|_{L^2}^2+2\gamma\| \nabla\delta b\|_{L^2}^2+2\gamma\int_{\mathbb{R}^2}\partial_t  \delta b\cdot  \delta b\dd x\\&~~~~~~~~~+\int_0^t\left(\| \nabla\delta u(\tau)\|_{L^2}^2+\| \nabla\delta b(\tau)\|_{L^2}^2+\gamma\|\partial_t  \delta b(\tau)\|_{L^2}^2\right)\dd t\\
	&\leq C\left(1+\gamma^\frac{1}{2}+\gamma^m\right)\int_0^t\left(\left\| (b^1,b^2)(\tau)\right\|_{H^{m+1}}^2+\left\| (u^1, u^2)(\tau)\right\|_{H^{m+1}}^2\right)\\&~~~~~~~~~~~~~~~~~~~~~~~~~~~~~~~~~\times\left(2\gamma^2\left\| \partial_t \delta b(\tau)\right\|_{L^2}^{2}+\left\| \delta b(\tau)\right\|_{L^2}^{2}+\left\|\delta u(\tau)\right\|_{L^2}^{2}\right)\dd \tau.
\end{align*}
The term  $2\gamma\int_{\mathbb{R}^2}\partial_t \delta b\cdot \delta b\dd x$ on the left side of the above inequality can be estimated as 
\begin{align*}
	2\gamma\int_{\mathbb{R}^2}\partial_t \delta b\cdot \delta b\dd x &\leq \dfrac{2\| \delta b\|_{L^2}^2}{3}+\dfrac{3\gamma^2\|\partial_t \delta b\|_{L^2}^2}{2}. 
\end{align*}
Then,
adopting to a  similar way in the proof of existence, we can deduce 
\begin{align*}
	&\|\delta u\|_{L^2}^2+\|  \delta b\|_{L^2}^2+2\gamma^2  \|\partial_t   \delta b\|_{L^2}^2\\
	\leq& C\left(1+\gamma^\frac{1}{2}+\gamma^m\right)\int_0^t\left(\left\| (b^1,b^2)(\tau)\right\|_{H^{m+1}}^2+\left\| (u^1, u^2)(\tau)\right\|_{H^{m+1}}^2\right)\\&~~~~~~~~~~~~~~~~~~~~~~~~~~~~~~~~~\times\left(2\gamma^2\left\| \partial_t \delta b(\tau)\right\|_{L^2}^{2}+\left\| \delta b(\tau)\right\|_{L^2}^{2}+\left\|\delta u(\tau)\right\|_{L^2}^{2}\right)\dd \tau.
\end{align*}
Since $(u^1, b^1)$ and $(u^2, b^2)$ are in the regularity class \eqref{th1-1}, we know that for any $t>0$, 
\begin{align*}
\int_{0}^t\left(\left\| (b^1,b^2)(\tau)\right\|_{H^{m+1}}^2+\left\| (u^1, u^2)(\tau)\right\|_{H^{m+1}}^2\right) \dd\tau \leq C(t)<\infty.
\end{align*}
By using  Gronwall's inequality, we obtain
\begin{align*}
	\delta u=\delta b= 0. 
\end{align*}
At this stage, we have proved the uniqueness of the solution in Theorem \ref{thm1}.
\end{proof}

\subsection{The decay estimates of $(u,b)$ in $L^q\times L^q$, for $q\geq 2$ }

\begin{proof}[Proof of the Theorem \ref{thm1}]
	{\bf  ($L^q$ decay estimates)} 
To prove the decay estimates of $(u,b)$, we will use continuity argument
(see, e.g. Tao (\cite{tao}, p20)).

Assume that 
\begin{align*}
	\left\| (u_0,b_0,a_0)\right\|_{X^m(\mr^2)}\define\|u_0\|_{H^m}+\|b_0\|_{H^{m+1}}+\|a_0\|_{H^m}\leq \varepsilon
\end{align*}
for sufficiently small $\varepsilon>0,$ and
 for any $t<T$,
\begin{align}\label{6.1}
	&	\|u(t)\|_{L^q}+	\|b(t)\|_{L^q}\leq C_0t^{\frac{1}{q}-\frac{1}{2}}\left(1+\gamma\right)\left\| (u_0,b_0,a_0)\right\|_{X^m(\mr^2)}.
\end{align}
Here the constant $C_0\geq C> 0$  is independent of $\gamma$ that will be specified later, and $C$ is the constant appearing on the right-hand side of \eqref{th1-1} in Theorem \ref{thm1}. 
By the continuity argument, it remains to verify 
\begin{align}\label{6.2}
	&	\|u(t)\|_{L^q}+	\|b(t)\|_{L^q}\leq \dfrac{C_0 }{2}t^{\frac{1}{q}-\frac{1}{2}}\left(1+\gamma\right)\left\| (u_0,b_0,a_0)\right\|_{X^m(\mr^2)}.
\end{align}

Taking the $L^q,~q\geq2$ norm of both sides of \eqref{u integral} yields	
\begin{align}\label{6.3}
	\|u(t)\|_{L^q}\leq \left\| \e^{t\Delta} u_0\right\|_{L^q}+\left\| \int_0^t \e^{(t-\tau)\Delta}\mathbb{P}\left(b\cdot\nabla b-u\cdot\nabla u\right)(\tau)\dd \tau \right\|_{L^q}.
\end{align}
 Utilizing Lemma \ref{heatkernel}, we have
\begin{align*}
	\left\| \e^{t\Delta} u_0\right\|_{L^q}&\leq  c_0 \left\| (u_0,b_0,a_0)\right\|_{X^m(\mr^2)}t^{\frac{1}{q}-\frac{1}{2}}.
\end{align*}

For the nonlinear term, we obtain
\begin{equation*}\begin{split}
	&\left\| \int_0^t \e^{(t-\tau)\Delta}\mathbb{P}\left(b\cdot\nabla b-u\cdot\nabla u\right)(\tau)\dd \tau \right\|_{L^q} \\\nonumber
	\leq& \int_0^\frac{t}{2} \left\|  \e^{(t-\tau)\Delta}\mathbb{P}\left(b\cdot\nabla b-u\cdot\nabla u\right)(\tau) \right\|_{L^q}\dd \tau +\int_\frac{t}{2}^t \left\|  \e^{(t-\tau)\Delta}\mathbb{P} \left(b\cdot\nabla b-u\cdot\nabla u\right)(\tau) \right\|_{L^q} \dd \tau \\
	\define&I_1+I_2.
\end{split}\end{equation*}
Denote $r$ as 
$
	\dfrac{1}{r}=\dfrac{1}{2}+\dfrac{1}{q},
$
then, we get by Lemma \ref{gag},
\begin{align*}
	\|u\|_{L^{2r}}\leq C\|u\|_{L^2}^\frac{1}{2}\|u\|_{L^q}^\frac{1}{2}.
\end{align*}
This, together with Lemma \ref{heatkernel}, \eqref{th1-1} and \eqref{6.1}  yields 
\begin{equation*}\begin{split}
	I_1
	&\leq C\int_0^\frac{t}{2}(t-\tau)^{-\frac{1}{2}-\frac{1}{r}+\frac{1}{q}}
	(\|u(\tau)\|_{L^{2r}}^2+	\|b(\tau)\|_{L^{2r}}^2)\dd \tau\\
	&\leq C\int_0^\frac{t}{2}(t-\tau)^{-1}\|(u,b)(\tau)\|_{L^2}\|(u,b)(\tau)\|_{L^q}\dd \tau\\
	&\leq Ct^{-1}\int_0^\frac{t}{2}\left(C_0\left(1+\gamma\right)\left\| (u_0,b_0,a_0)\right\|_{X^m(\mr^2)} \tau^{\frac{1}{q}-\frac{1}{2}}\right)\\
	&~~~~~~\left(C\left(1+\gamma^{\frac{m}{2m+2}}\right)\left\| (u_0,b_0,a_0)\right\|_{X^m(\mr^2)}\right)\dd\tau\\
	&\leq C\left(1+\gamma\right)\left(1+\gamma^{\frac{m}{2m+2} }\right)\left(C_0\left\| (u_0,b_0,a_0)\right\|_{X^m(\mr^2)}\right)^2 t^{\frac{1}{q}-\frac{1}{2}}.
\end{split}\end{equation*}  $C_0\geq C$  has been used in the last inequality, where $C$ is the constant appearing on the right-hand side of \eqref{th1-1} in Theorem \ref{thm1}.

As for $I_2$, using \eqref{6.1} and	Lemma \ref{heatkernel} again,  we get for $q>2,$
\begin{align*}
	I_2
	&\leq C\int_{\frac{t}{2}}^{t}(t-\tau)^{-\frac{1}{2}-\frac{1}{q}}\left(\|u(\tau)\|_{L^{q}}^2+\|b(\tau)\|_{L^{q}}^2\right)\dd\tau\\
	&\leq C\int_\frac{t}{2}^t (t-\tau)^{-\frac{1}{2}-\frac{1}{q}}\left(C_0\left(1+\gamma\right)\left\| (u_0,b_0,a_0)\right\|_{X^m(\mr^2)} \tau^{\frac{1}{q}-\frac{1}{2}}\right)^2\dd\tau\\
	&\leq C\left(1+\gamma\right)^2\left(C_0\left\| (u_0,b_0,a_0)\right\|_{X^m(\mr^2)} \right)^2 t^{\frac{2}{q}-1}\int_{\frac{t}{2}}^{t}(t-\tau)^{-\frac{1}{2}-\frac{1}{q}}\dd\tau\\
	&\leq C\left(1+\gamma\right)^2\left(C_0\left\| (u_0,b_0,a_0)\right\|_{X^m(\mr^2)} \right)^2 t^{\frac{1}{q}-\frac{1}{2}}.
\end{align*}
By inserting  these estimates into \eqref{6.3}, we  get
\begin{align}\nonumber
	&\|u(t)\|_{L^q}\leq c_0\left\| (u_0,b_0,a_0)\right\|_{X^m(\mr^2)} t^{\frac{1}{q}-\frac{1}{2}}\\\label{6.5}
	&~~~~~~~~~~~~~~~+C\left(1+\gamma\right)\left(1+\gamma^{\frac{m}{2m+2}}+\gamma\right)\left(C_0\left\| (u_0,b_0,a_0)\right\|_{X^m(\mr^2)} \right)^2 t^{\frac{1}{q}-\frac{1}{2}}.
\end{align}

Taking the $L^q$ norm of both sides of \eqref{b integral} yields
\begin{equation}\label{b}
	\begin{split}
&	\|b(t)\|_{L^q}\\ \leq& \left\|\Big(K_0+\frac{1}{2}K_1\Big)b_0\right\|_{L^q}+\gamma\left\|K_1a_0\right\|_{L^q}+\left\|\int_{0}^{t}K_1(t-\tau)\left(b\cdot\nabla u-u\cdot\nabla b\right)(\tau)\dd \tau\right\|_{L^q}\\
	\define &J_1+ J_2+ J_3.\end{split}
\end{equation}
Due to the different behavior of the operator $K_0$ and $K_1$ at different frequencies, we  split the estimate of $J_1$ into  two parts. Indeed, by using Lemma \ref{hausdorff-Young}, we have
\begin{align*}
	J_1&\leq \left\|\left(K_0+\frac{1}{2}K_1\right)b_0\right\|_{L^q(\mathbb{R}^2)}\\&\leq	\left\|\w{\left(K_0+\frac{1}{2}K_1\right)}(\xi)b_0\right\|_{L^{q'}(S_1)}+	\left\|\w{\left(K_0+\frac{1}{2}K_1\right)}(\xi)b_0\right\|_{L^{q'}(S_2)}\\
	&\stackrel{\mathrm{def}}{=}	J_{11}+	J_{12},
\end{align*}where $\frac{1}{q'}=1-\frac{1}{q}.$
Applying Lemma \ref{young+KDf} and \eqref{fren-1} gives to, for some $p$ satisfying $p>2, ~\frac{1}{p}+\frac{1}{q}=\frac{1}{2},$
\begin{align*}
J_{11}&\leq C\left\||\xi|^{-1}\w{\left(K_0+\frac{1}{2}K_1\right)}(\xi)\right\|_{L^p(S_1)}\|\Lambda b_0\|_{L^2(\mr^2)}\\
	&\leq C\lef |\xi|^{-1}\e^{-\frac{1}{8\gamma}t} \rig_{L^p(\{\xi;4\gamma |\xi|^2\geq \frac{3}{4}\})}\|b_0\|_{H^1(\mr^2)}\\
	&\leq C\e^{-\frac{1}{8\gamma}t}\left(\int_{\frac{\sqrt{3}}{4\sqrt{\gamma}}}^{+\infty}r^{-p+1}\dd r\right)^{\frac{1}{p}}\|b_0\|_{H^1(\mr^2)}\\
	&\leq C\e^{-\frac{1}{8\gamma}t}\left(\frac{1}{p-2}(\frac{\sqrt{3}}{4\sqrt{\gamma}})^{2-p}\right)^{\frac{1}{p}}\|b_0\|_{H^1(\mr^2)}\\
	&\leq C\gamma^{\frac{1}{2}-\frac{1}{p}}\e^{-\frac{1}{8\gamma}t}\|b_0\|_{H^1(\mr^2)}\\&\leq C\gamma^{\frac{1}{2}}t^{\frac{1}{q}-\frac{1}{2}}\|b_0\|_{H^{1+m}(\mr^2)},
\end{align*}
where in the last inequality we have used the  following fact:
\begin{align}\label{basic}
	\alpha \geq 0\Rightarrow t^\alpha\e^{-\frac{1}{8\gamma}t}=\left(8\gamma\right)^\alpha\left(\dfrac{t}{8\gamma}\right)^\alpha\e^{-\frac{1}{8\gamma}t}\leq C\gamma^\alpha.
\end{align}
For the $J_{12}$ part,  utilizing  Lemma \ref{young+KDf} and \eqref{fren-3} yields, for some $p$ satisfying $p>2, ~\frac{1}{p}+\frac{1}{q}=\frac{1}{2},$
\begin{equation*}\begin{split}
J_{12}&\leq C\left\|\w{\left(K_0+\frac{1}{2}K_1\right)}(\xi)\right\|_{L^p(S_2)}\|b_0\|_{L^2(\mr^2)}\\
	&\leq C\left\|\e^{-|\xi|^2t}\right\|_{L^p(S_2)}\|b_0\|_{L^2(S_2)}\leq Ct^{\frac{1}{q}-\frac{1}{2}}\|b_0\|_{L^2(\mr^2)}.
\end{split}\end{equation*}
It follows that by collecting the above estimates, 
\begin{equation}\label{b-Lq0}
	\begin{split}
	J_{1}
	&\leq C t^{\frac{1}{q}-\frac{1}{2}}\left(\gamma^{\frac{1}{2}}\|b_0\|_{H^{m+1}(\mr^2)}+\|b_0\|_{L^2(\mr^2)}\right)\\
	&\leq Ct^{\frac{1}{q}-\frac{1}{2}}\left(1+\gamma^\frac{1}{2}\right)\|b_0\|_{H^{m+1}(\mr^2)}\\
	&\leq c_1\left(1+\gamma\right)\left\| (u_0,b_0,a_0)\right\|_{X^m(\mr^2)} t^{\frac{1}{q}-\frac{1}{2}}.
\end{split}\end{equation}

As for $J_{2}$, 
\begin{align*}
	J_{2}=\gamma\left\|K_1a_0\right\|_{L^q(\mr^2)}&\leq\gamma\left\|\w{K_1a_0}\right\|_{L^{q'}(S_1)}+\gamma\left\|\w{K_1a_0}\right\|_{L^{q'}(S_2)} \\ & \define J_{21}+J_{22}.
\end{align*}
According to  Lemma \ref{young+KDf}, \eqref{fren-2} and \eqref{basic},  we obtain
\begin{align*}
	J_{21}&\leq C\gamma \lef \gamma^{-\frac{1}{2}}|\xi|^{-1}\e^{-\frac{1}{8\gamma}t} \rig_{L^p(\{\xi;4\gamma |\xi|^2\geq \frac{3}{4}\})}\|a_0\|_{L^2(\mr^2)}\\
	&\leq C\gamma^{\frac{1}{2}}\e^{-\frac{1}{8\gamma}t}\left(\int_{\frac{\sqrt{3}}{4\sqrt{\gamma}}}^{+\infty}r^{-p+1}\dd r\right)^{\frac{1}{p}}\|a_0\|_{L^2(\mr^2)}\\
	&\leq C\gamma^{\frac{1}{2}}\e^{-\frac{1}{8\gamma}t}\left(\frac{1}{p-2}(\frac{\sqrt{3}}{4\sqrt{\gamma}})^{2-p}\right)^{\frac{1}{p}}\|a_0\|_{L^2(\mr^2)}\\
	&\leq C\gamma^{1-\frac{1}{p}}\gamma^{-\frac{1}{q}+\frac{1}{2}}t^{\frac{1}{q}-\frac{1}{2}}\|a_0\|_{L^2(\mr^2)}\leq C\gamma t^{\frac{1}{q}-\frac{1}{2}}\|a_0\|_{L^2(\mr^2)}.
\end{align*}	
 $J_{22}$ can be estimated in a similar fashion as $J_{12}.$ Indeed, we get
\begin{align*}
	J_{22}&\leq C\gamma t^{\frac{1}{q}-\frac{1}{2}}\|a_0\|_{L^2(\mr^2)}.
\end{align*}
As a result, we have
\begin{align}\label{6.12}
	J_{2}\leq C\gamma t^{\frac{1}{q}-\frac{1}{2}}\|a_0\|_{L^2(\mr^2)}\leq c_2\gamma\left\| (u_0,b_0,a_0)\right\|_{X^m(\mr^2)} t^{\frac{1}{q}-\frac{1}{2}}.
\end{align}

Similarly, we  split the estimate of the nonlinear term $J_{3}$ into two parts
\begin{equation}\label{b-Lq-1}
	\begin{split}
	J_{3}
	&\leq \left\|\int_{0}^{t}\w{K_1(t-\tau)\nabla\cdot\left(u\otimes b\right)}\dd \tau\right\|_{L^{q'}(S_1)}+\left\|\int_{0}^{t}\w{K_1(t-\tau)\nabla\cdot\left(u\otimes b\right)}\dd \tau\right\|_{L^{q'}(S_2)}\\
	&\stackrel{\mathrm{def}}{=}J_{31}+J_{32}.
\end{split}
\end{equation}
Applying Lemma \ref{young+KDf} and \eqref{fren-2} yields that
\begin{equation*}\begin{split}
	J_{31}&=\left\|\int_{0}^{t}\w{K_1(t-\tau)\nabla\cdot\left(u\otimes b\right)}\dd \tau\right\|_{L^{q'}(S_1)}\\
	&\leq C\int_{0}^{t} \lef |\xi|^{-m} \w{K_1}(\xi,t-\tau)\rig_{L^{2}(S_1)}\lef \Lambda^{m}\nabla\cdot\left(u(\tau)\otimes b(\tau)\right)\rig_{L^{\frac{2q}{2+q}}(\mr^2)}\dd \tau\\
	&\leq C \int_{0}^t \lef \gamma^{-\frac{1}{2}}|\xi|^{-1-m}\e^{-\frac{1}{8\gamma}(t-\tau)} \rig_{L^{2}(\{\xi;4\gamma |\xi|^2\geq \frac{3}{4}\})} \left\|\Lambda^{m}\nabla\cdot\left(u(\tau)\otimes b(\tau)\right)\right\|_{L^{\frac{2q}{2+q}}(\mr^2)} \dd\tau.
\end{split}\end{equation*}
Combining the above estimate with the following two basic calculations
\begin{align*}
	\left\||\xi|^{-1-m}\right\|_{L^{2}(\{\xi;4\gamma |\xi|^2\geq \frac{3}{4}\})}=\left(\int_{\frac{\sqrt{3}}{4\sqrt{\gamma}}}^{+\infty}\left(r^{-2m-1}\right)\dd r \right)^{\frac{1}{2}}=(\frac{1}{2m})^\frac12\left(\frac{\sqrt{3}}{4\sqrt{\gamma}}\right)^{-m},
\end{align*}
and
\begin{align*}
	&\left\|\Lambda^{m}\nabla\cdot\left(u(\tau)\otimes b(\tau)\right)\right\|_{L^{\frac{2q}{2+q}}(\mr^2)}\\ \leq &C \left(\|\nabla\Lambda^{m}u(\tau)\|_{L^{2}}\|b(\tau)\|_{L^{q}}+\|\nabla\Lambda^{m}b(\tau)\|_{L^{2}}\|u(\tau)\|_{L^q}\right),
\end{align*}
further implies that, after exploiting \eqref{p-3} in Lemma \ref{expintegral}
\begin{equation}\label{b-Lq2}
	\begin{split}
	J_{31}&\leq C\gamma^{\frac{m-1}{2}}\int_{0}^{t}\e^{-\frac{1}{8\gamma}(t-\tau)}\left(\|\nabla\Lambda^{m}u(\tau)\|_{L^{2}}\|b(\tau)\|_{L^{q}}+\|\nabla\Lambda^{m}b(\tau)\|_{L^{2}}\|u(\tau)\|_{L^q}\right)\dd\tau\\
	&\leq C\gamma^{\frac{m-1}{2}}\int_{0}^{t}\e^{-\frac{1}{8\gamma}(t-\tau)}\left(C_0\left(1+\gamma\right)\left\| (u_0,b_0,a_0)\right\|_{X^m(\mr^2)} \tau^{\frac{1}{q}-\frac{1}{2}}\right)\lef \nabla(u,b)(\tau)\rig_{H^m}\dd\tau\\
	&\leq C\gamma^{\frac{m-1}{2}}C_0\left(1+\gamma\right)\left\| (u_0,b_0,a_0)\right\|_{X^m(\mr^2)}\\
	&~~~~~~~~~~~~~~~\left(\int_{0}^t\e^{-\frac{1}{4\gamma}(t-\tau)}\tau^{\frac{2}{q}-1}\dd\tau\right)^{\frac{1}{2}}\left(\int_{0}^t\lef \nabla(u,b)(\tau)\rig_{H^m}^2\dd\tau\right)^{\frac{1}{2}}\\
\\	&\leq C\left(1+\gamma\right)\left(\gamma^{\frac{m}{2}}+\gamma^{\frac{m}{2}+\frac{m}{2m+2}}\right)\left(C_0\left\| (u_0,b_0,a_0)\right\|_{X^m(\mr^2)} \right)^2t^{\frac{1}{q}-\frac{1}{2}}.
\end{split}\end{equation}
To differentiate the singular behavior of $\tau$ near $0$ and $t$, $J_{32}$
is divided  into two parts,
\begin{equation*}
	\begin{split}
		J_{32}&=		\left\|\int_{0}^{t}\w{K_1(t-\tau)\nabla\cdot\left(u\otimes b\right)}\dd \tau\right\|_{L^{q'}(S_2)}\\
	&\leq \int_{0}^\frac{t}{2}\left\|\w{K_1(t-\tau)\nabla\cdot\left(u\otimes b\right)}\right\|_{L^{q'}(S_2)}\dd \tau \\
	&~~~+\int_{\frac{t}{2}}^{t}\left\|\w{K_1(t-\tau)\nabla\cdot\left(u\otimes b\right)}\right\|_{L^{q'}(S_2)}\dd \tau \\
	&\define 	\overline{J_{32}}+	\widetilde{J_{32}}.
\end{split}
\end{equation*}
Thanks to Lemma \ref{young+KDf}, $\overline{J_{32}}$ is bounded by
\begin{equation}\label{b-Lq3}
	\begin{split}
		\overline{J_{32}}&\leq C\int_{0}^\frac{t}{2} \||\xi|\e^{-|\xi|^2(t-\tau)}\|_{L^2(S_2)} \left\| u(\tau)\otimes b(\tau)\right\|_{L^{\frac{2q}{2+q}}} \dd\tau\\
	&\leq C\int_{0}^\frac{t}{2} (t-\tau)^{-1} \|u(\tau)\|_{L^2}\|b(\tau)\|_{L^q}\dd\tau\\
	&\leq Ct^{-1}\left(1+\gamma\right)\left(1+\gamma^\frac{m}{2m+2}\right)\int_{0}^\frac{t}{2} \left(C_0\left\| (u_0,b_0,a_0)\right\|_{X^m(\mr^2)}\right)^2 \tau^{\frac{1}{q}-\frac{1}{2}}\dd\tau\\
	&\leq C\left(1+\gamma\right)\left(1+\gamma^\frac{m}{2m+2}\right)\left(C_0\left\| (u_0,b_0,a_0)\right\|_{X^m(\mr^2)}\right)^2 t^{\frac{1}{q}-\frac{1}{2}}.
\end{split}
\end{equation}
As for $\widetilde{J_{32}},$ by using Lemma \ref{young+KDf}, we first get
\begin{equation*}\begin{split}
		\widetilde{J_{32}}
		&\leq  C \int_{\frac{t}{2}}^{t}\lef|\xi|\e^{-|\xi|^2(t-\tau)}\rig_{L^p(S_2)}\lef u(\tau)\otimes b(\tau)\right\|_{L^k(\mr^2)}\dd \tau,
	\end{split}
	\end{equation*}
where $\dfrac{1}{q}+\dfrac{1}{p}=\dfrac{1}{\kappa}.$
By choosing $p=\frac{q}{2}+1$,  using H\"{o}lder's inequality and
\eqref{6.1},  we have
\begin{equation}\label{b-Lq4}
	\begin{split}
	\widetilde{J_{32}}
&\leq  C\int_{\frac{t}{2}}^{t}(t-\tau)^{-\frac{1}{2}-\frac{2}{q+2}}\|u(\tau)\|_{L^q(\mr^2)}\|b(\tau)\|_{L^{\frac{q}{2}+1}(\mr^2)}\dd\tau\\&\leq C\left(1+\gamma\right)\left(1+\gamma^{\frac{1}{2}+\frac{m}{4m+4}}\right)\left(C_0\left\| (u_0,b_0,a_0)\right\|_{X^m(\mr^2)}\right)^2\\&~~~~~~~~~~~~~~~\times \int_{\frac{t}{2}}^{t}(t-\tau)^{-\frac{1}{2}-\frac{2}{q+2}}\tau^{\frac{1}{q}-\frac{1}{2}}\tau^{\frac{2}{q+2}-\frac{1}{2}}\dd\tau\\&\leq C\left(1+\gamma\right)\left(1+\gamma^{\frac{1}{2}+\frac{m}{4m+4}}\right)\left(C_0\left\| (u_0,b_0,a_0)\right\|_{X^m(\mr^2)}\right)^2 t^{\frac{1}{q}-\frac{1}{2}}.
\end{split}
\end{equation}
Inserting \eqref{b-Lq2},  \eqref{b-Lq3} and  \eqref{b-Lq4}  into  \eqref{b-Lq-1} gives to
\begin{equation}\label{b-Lq}
	\begin{split}
		J_{3}
		&\leq C\left(1+\gamma\right)\left(1+\gamma^{\frac{1}{2}+\frac{m}{4m+4}}+\gamma^\frac{m}{2m+2}\right)\left(C_0\left\| (u_0,b_0,a_0)\right\|_{X^m(\mr^2)}\right)^2t^{\frac{1}{q}-\frac{1}{2}}.
	\end{split}
\end{equation}
Combining \eqref{b-Lq0}, \eqref{6.12}, \eqref{b-Lq} and \eqref{b}, and applying Young's inequality to the polynomial of $\gamma$ yields
\begin{equation}\label{b-Lq-f}
	\begin{split}
	\|b(t)\|_{L^q}&\leq \left(c_1+c_2\right)\left(1+\gamma\right)\left\| (u_0,b_0,a_0)\right\|_{X^m(\mr^2)} t^{\frac{1}{q}-\frac{1}{2}}\\
	&~~~~+C\left(1+\gamma\right)\left(1+\gamma^{\frac{1}{2}+\frac{m}{4m+4}}+\gamma^{\frac{m}{2}+\frac{m}{2m+2}}\right)\left(C_0\left\| (u_0,b_0,a_0)\right\|_{X^m(\mr^2)}\right)^2t^{\frac{1}{q}-\frac{1}{2}}.
	\end{split}
\end{equation}

It follows that by adding \eqref{6.5} and \eqref{b-Lq-f} together
\begin{equation*}\begin{split}
	&\|u(t)\|_{L^q}+	\|b(t)\|_{L^q}\\ \leq& \left(c_0+c_1+c_2\right)\left\| (u_0,b_0,a_0)\right\|_{X^m(\mr^2)}\left(1+\gamma\right) t^{\frac{1}{q}-\frac{1}{2}}\\
	&~~~+C\left(1+\gamma\right)\left(1+\gamma+\gamma^{\frac{1}{2}+\frac{m}{4m+4}}+\gamma^{\frac{m}{2}+\frac{m}{2m+2}}\right)\left(C_0\left\| (u_0,b_0,a_0)\right\|_{X^m(\mr^2)}\right)^2t^{\frac{1}{q}-\frac{1}{2}}.
\end{split}\end{equation*}
which implies \eqref{6.2}, by selecting $C_0\geq 4\left(c_0+c_1+c_2\right)$ and a sufficiently small initial value  such that \begin{align*}
	CC_0\left\| (u_0,b_0,a_0)\right\|_{X^m(\mr^2)}\left(1+\gamma+\gamma^{\frac{1}{2}+\frac{m}{4m+4}}+\gamma^{\frac{m}{2}+\frac{m}{2m+2}}\right)\leq \frac{1}{4}.
\end{align*}
Here, we assume that $\left\|(u_0,b_0,a_0)\right\|_{X^m(\mr^2)}\leq \varepsilon$, and \begin{align*}
	\varepsilon \leq  \dfrac{1}{4CC_0\left(1+\gamma+\gamma^{\frac{1}{2}+\frac{m}{4m+4}}+\gamma^{\frac{m}{2}+\frac{m}{2m+2}}\right)}.
\end{align*}
At this point, we have completed the entire proof of Theorem \ref{thm1}.

\end{proof}

\section{The Decay of $(u, b)$ in $\dot{H}^\beta\times \dot{H}^\varrho$, for $0\leq \beta\leq m, 0\leq\varrho<m+1$ }

\begin{proof}[Proof of the Theorem \ref{thm2}]
	Since global existence and uniqueness have already been proved in Theorem \ref{thm1}, we first assume that $t < 1$, and then we get
	\begin{align*}
		\|\Lambda^\beta u(t)\|_{L^2}&\leq C\left(1+\gamma^{\frac{m}{2m+2}}\right)\left\| (u_0,b_0,a_0)\right\|_{X^{m,c}(\mr^2)} \\
		&\leq C\left(1+\gamma^{1+\frac{1}{c}-\frac{1-\beta}{2}}\right)\left\| (u_0,b_0,a_0)\right\|_{X^{m,c}(\mr^2)}  (1+t)^{\frac{1-\beta}{2}-\frac{1}{c}},\\
		\|\Lambda^\varrho b(t)\|_{L^2}&\leq C\left(1+\gamma^{\frac{m}{2m+2}}\right)\left\| (u_0,b_0,a_0)\right\|_{X^m(\mr^2)}\\
		&\leq C\left(1+\gamma^{1+\frac{1}{c}-\frac{1-\beta}{2}}\right)\left\| (u_0,b_0,a_0)\right\|_{X^{m,c}(\mr^2)} (1+t)^{\frac{1-\varrho}{2}-\frac{1}{c}},
	\end{align*}
	where $0\leq \beta \leq m$ and $0 \leq \varrho < m + 1$.
	
	Based on the above statements, we only need to prove the case when $t \geq  1$ in the following.
	For the sake of clarity, we divide the estimates into two decay levels.
 In the first part, we will prove \eqref{th2-1}. In the second part, we will prove \eqref{th2-2}.
	
	{\bf Step 1.} 
Assume that 
\begin{align*}
	\|(u_0,b_0,a_0)\|_{X^{m,c}(\mr^2)}\define \|u_0\|_{H^m}+\|b_0\|_{H^{m+1}}+\|a_0\|_{H^m}+\|u_0\|_{L^c}+\|b_0\|_{L^c}+\|a_0\|_{L^c}\leq \varepsilon
\end{align*}
for suffficiently small $\varepsilon>0$, and
for any $t<T$, $0\leq \beta\leq m,$
	\begin{align}\label{high-1}
	\|\Lambda^\beta u(t)\|_{L^2}+	\|\Lambda^\beta b(t)\|_{L^2}\leq C_0\left(1+\gamma^{1+\frac{1}{c}-\frac{1-\beta}{2}}\right)	\|(u_0,b_0,a_0)\|_{X^{m,c}(\mr^2)} (1+t)^{\frac{1-\beta}{2}-\frac{1}{c}}.
\end{align}
Here $C_0\geq C> 0$  is a constant that will be specified later, where $C$ is the constant appearing on the right-hand side of \eqref{th1-1} in Theorem \ref{thm1}, independent of $\gamma$.
 By the continuity argument (\cite{tao}), it remains to verify 
	\begin{equation}\label{high-2}
		\begin{split}
		&	\|\Lambda^\beta u(t)\|_{L^2}+	\|\Lambda^\beta b(t)\|_{L^2}\leq \frac{C_0}{2} \left(1+\gamma^{1+\frac{1}{c}-\frac{1-\beta}{2}}\right)	\|(u_0,b_0,a_0)\|_{X^{m,c}(\mr^2)}(1+t)^{\frac{1-\beta}{2}-\frac{1}{c}}.\end{split}
	\end{equation}

	
First, it follows directly from  \eqref{high-1} that
	\begin{align}\label{high-0}
		\|\Lambda^\beta u(t)\|_{L^2}+	\|\Lambda^\beta b(t)\|_{L^2}\leq C_0\left(1+\gamma^{1+\frac{1}{c}-\frac{1-\beta}{2}}\right)	\|(u_0,b_0,a_0)\|_{X^{m,c}(\mr^2)} t^{\frac{1-\beta}{2}-\frac{1}{c}}.	\end{align}	
	
	Taking the $\dot{H}^{\beta}$ norm of the integral form of the solution $u(x,t)$ as shown in \eqref{u integral} of Proposition \ref{integral}, we can deduce that
	\begin{align*}
		\|\Lambda^\beta u\|_{L^2}\leq \left\|\Lambda^\beta \e^{t\Delta} u_0\right\|_{L^2}+\left\| \int_0^t \e^{(t-\tau)\Delta}\Lambda^\beta\mathbb{P}\left(b\cdot\nabla b-u\cdot\nabla u\right)\dd \tau \right\|_{L^2},
	\end{align*}
	where $0\leq \beta\leq m$. 
	 Applying Lemma \ref{heatkernel} yields
	\begin{align}\label{h-0}
		\left\| \Lambda^\beta \e^{t\Delta} u_0\right\|_{L^2}&\leq Ct^{\frac{1-\beta}{2}-\frac{1}{c}}\|u_0\|_{L^c}\leq c_0	\|(u_0,b_0,a_0)\|_{X^{m,c}(\mr^2)}(1+t)^{\frac{1-\beta}{2}-\frac{1}{c}}.
	\end{align}

	To distinguish the effects caused by different singularities, we divide the estimates of the nonlinear term into the following two parts:
	\begin{align*}
		&\left\| \int_0^t \e^{(t-\tau)\Delta}\Lambda^\beta\mathbb{P}\left(b\cdot\nabla b-u\cdot\nabla u\right)\dd \tau \right\|_{L^2} \\
		\leq& \int_0^\frac{t}{2} \left\|  \e^{(t-\tau)\Delta}\Lambda^\beta\mathbb{P}\left(b\cdot\nabla b-u\cdot\nabla u\right) \right\|_{L^2}\dd \tau +\int_\frac{t}{2}^t \left\|  \e^{(t-\tau)\Delta}\Lambda^\beta\mathbb{P} \left(b\cdot\nabla b-u\cdot\nabla u\right) \right\|_{L^2} \dd \tau \\
		=&H_1+H_2.
	\end{align*}

	For $H_1$, by Lemma \ref{heatkernel}, \eqref{th1-1} and \eqref{high-1}, we obtain
	\begin{equation}\label{h-1}
		\begin{split}
	H_1&=\int_0^\frac{t}{2} \left\|  \e^{(t-\tau)\Delta}\Lambda^\beta\mathbb{P}\left(b\cdot\nabla b-u\cdot\nabla u\right)(\tau) \right\|_{L^2}\dd \tau \\
		&\leq C\int_0^\frac{t}{2}(t-\tau)^{-1-\frac{\beta}{2}}\|u(\tau),b(\tau)\|_{L^2}^2\dd \tau\\ &\leq C\left(1+\gamma^{1+\frac{1}{c}-\frac{1}{2}}\right)\left(1+\gamma^{\frac{m}{2m+2}}\right)\left(C_0\|(u_0,b_0,a_0)\|_{X^{m,c}(\mr^2)}\right)^2t^{-1-\frac{\beta}{2}}\int_0^\frac{t}{2} \tau^{\frac{1}{2}-\frac{1}{c}}\dd\tau\\
		&\leq C\left(1+\gamma^{1+\frac{1}{c}-\frac{1}{2}}\right)\left(1+\gamma^{\frac{m}{2m+2}}\right)\left(C_0\|(u_0,b_0,a_0)\|_{X^{m,c}(\mr^2)}\right)^2 t^{\frac{1-\beta}{2}-\frac{1}{c}}\\
		&\leq C\left(1+\gamma^{1+\frac{1}{c}-\frac{1-\beta}{2}}\right)\left(1+\gamma^{\frac{m}{2m+2}}\right)\left(C_0\|(u_0,b_0,a_0)\|_{X^{m,c}(\mr^2)}\right)^2 (1+t)^{\frac{1-\beta}{2}-\frac{1}{c}}.
	\end{split}
\end{equation}

	When $0\leq\beta<1$ (either $m < 1$ and $0\leq  \beta \leq m$, or $m \geq  1$ and $0\leq\beta<1$), we invoke  Lemma \ref{heatkernel} and Sobolev embedding inequality to deduce
\begin{equation}\label{high-}
	\begin{split}
		H_2
		&\leq C\int_\frac{t}{2}^t (t-\tau)^{-1+\frac{\beta}{2}}\left\|\Lambda^\beta \left(b\otimes b+u\otimes u\right)\right\|_{L^\frac{2}{2-\beta}}\dd\tau\\
		&\leq C\int_{\frac{t}{2}}^t(t-\tau)^{-1+\frac{\beta}{2}} \|u(\tau),b(\tau)\|_{L^\frac{2}{1-\beta}}\left\|\Lambda^\beta\left(u(\tau),b(\tau)\right)\right\|_{L^{2}}\dd\tau\\
		&\leq C\int_{\frac{t}{2}}^t(t-\tau)^{-1+\frac{\beta}{2}} \left\|\Lambda^\beta\left(u(\tau),b(\tau)\right)\right\|_{L^{2}}^2\dd\tau\\
		&\leq C\int_{\frac{t}{2}}^t(t-\tau)^{-1+\frac{\beta}{2}} \left\|\Lambda^\beta\left(u(\tau),b(\tau)\right)\right\|_{L^{2}}^{1+\theta} \left\|\left(u(\tau),b(\tau)\right)\right\|_{H^m}^{1-\theta}\dd\tau,
		\end{split}
\end{equation} where $\theta=\dfrac{\frac{\beta}{2}}{\frac{\beta}{2}+\frac{1}{c}-\frac{1}{2}}.$
It is noted that
	\begin{equation*}\begin{split}
		&\left\|\Lambda^\beta\left(u(\tau),b(\tau)\right)\right\|_{L^{2}}^{1+\theta}\\
		\leq&\left(C_0\|(u_0,b_0,a_0)\|_{X^{m,c}(\mr^2)} \right)^{1+\theta}\left(1+\gamma^{1+\frac{1}{c}-\frac{1-\beta}{2}}\right)^{1+\theta}\tau^{\frac{1-\beta}{2}-\frac{1}{c}-\frac{\beta}{2}}.
\end{split}	\end{equation*}
	Substituting it back into \eqref{high-} and applying  Young's inequality to $\left(1+\gamma^{\frac{m}{2m+2}}\right)^{1-\theta}$, it holds that
	\begin{equation}\label{h-2}
			\begin{split}
		H_2&\leq C\left(C_0\|(u_0,b_0,a_0)\|_{X^{m,c}(\mr^2)}\right)^{2}\\
		&~~~~~~\left(1+\gamma^{\frac{m}{2m+2}}\right)^{1-\theta}\left(1+\gamma^{1+\frac{1}{c}-\frac{1-\beta}{2}}\right)^{1+\theta}\int_{\frac{t}{2}}^t (t-\tau)^{-1+\frac{\beta}{2}} \tau^{\frac{1-\beta}{2}-\frac{1}{c}-\frac{\beta}{2}}\dd\tau\\
		&\leq C\left(C_0\|(u_0,b_0,a_0)\|_{X^{m,c}(\mr^2)}\right)^2\left(1+\gamma^{1+\frac{1}{c}-\frac{1-\beta}{2}}\right)^2 t^{\frac{1-\beta}{2}-\frac{1}{c}-\frac{\beta}{2}} \int_{\frac{t}{2}}^t (t-\tau)^{-1+\frac{\beta}{2}}\dd\tau\\
		&\leq C\left(1+\gamma^{1+\frac{1}{c}-\frac{1-\beta}{2}}\right)^2\left(C_0\|(u_0,b_0,a_0)\|_{X^{m,c}(\mr^2)}\right)^2 t^{\frac{1-\beta}{2}-\frac{1}{c}}\\
		&\leq C\left(1+\gamma^{1+\frac{1}{c}-\frac{1-\beta}{2}}\right)^2\left(C_0\|(u_0,b_0,a_0)\|_{X^{m,c}(\mr^2)}\right)^2 (1+t)^{\frac{1-\beta}{2}-\frac{1}{c}}.
	\end{split}
	\end{equation}

	On the other hand, when $\beta\geq 1$, taking $1<\kappa<2$,
	\begin{equation}\label{high-3}
		\begin{split}
		H_2
		&\leq C\int_\frac{t}{2}^t (t-\tau)^{-\frac{1}{\kappa}}\left\|\Lambda^\beta \left(b\otimes b+u\otimes u\right)\right\|_{L^\kappa}\dd\tau\\
		&\leq C\int_\frac{t}{2}^t (t-\tau)^{-\frac{1}{\kappa}}\|u(\tau),b(\tau)\|_{L^\frac{2\kappa}{2-\kappa}}\left\|\Lambda^\beta\left(u(\tau),b(\tau)\right)\right\|_{L^{2}}\dd\tau.
	\end{split}\end{equation}
	Since $1\leq \beta \leq m$ and $0<2-\frac{2}{\kappa}<1\leq m$, we derive 
	\begin{align*}
		&\|u(\tau),b(\tau)\|_{L^\frac{2\kappa}{2-\kappa}}\left\|\Lambda^\beta\left(u(\tau),b(\tau)\right)\right\|_{L^{2}}\\
		&\leq C \|\Lambda^{\frac{2\kappa-2}{\kappa}}\left(u(\tau),b(\tau)\right)\|_{L^2}\left\|\Lambda^\beta\left(u(\tau),b(\tau)\right)\right\|_{L^{2}}\\
		&\leq C\|\Lambda^{\frac{2\kappa-2}{\kappa}}\left(u(\tau),b(\tau)\right)\|_{L^2}^{\theta_1}\left\| u(\tau),b(\tau)\right\|_{H^m}^{1-\theta_1}\left\|\Lambda^\beta\left(u(\tau),b(\tau)\right)\right\|_{L^{2}},
	\end{align*}
	with $\theta_1=\dfrac{1-\frac{1}{\kappa}}{\frac{1}{c}+\frac{1}{2}-\frac{1}{\kappa}}.$ By choosing $\beta=\frac{2k-2}{k}$ in \eqref{high-0}, we have 
	\begin{equation*}\begin{split}
	\|\Lambda^{\frac{2k-2}{k}}\left(u(\tau),b(\tau)\right)\|_{L^2}\leq C_0\left(1+\gamma^{\frac{3}{2}+\frac{1}{c}-\frac{1}{\kappa}}\right)\|(u_0,b_0,a_0)\|_{X^{m,c}(\mr^2)} \tau^{\frac{2-\kappa}{2\kappa}-\frac{1}{c}}.
\end{split}	
\end{equation*}
This, together with \eqref{th1-1} and \eqref{high-0} yields
	\begin{align*}
		&\|u(\tau),b(\tau)\|_{L^\frac{2\kappa}{2-\kappa}}\left\|\Lambda^\beta\left(u(\tau),b(\tau)\right)\right\|_{L^{2}}\\
		\leq& C\left(1+\gamma^{1+\frac{1}{c}-\frac{1-\beta}{2}}\right)\left(1+\gamma^{\frac{3}{2}+\frac{1}{c}-\frac{1}{\kappa}}\right)^{\theta_1}\left(1+\gamma^{\frac{m}{2m+2}}\right)^{1-\theta_1}\\
		&~~~~~~ \left(C_0\|(u_0,b_0,a_0)\|_{X^{m,c}(\mr^2)}\right)^2\tau^{\frac{1-\beta}{2}-\frac{1}{c}}\tau^{\left(\frac{2-\kappa}{2\kappa}-\frac{1}{c}\right)\theta_1}\\
		\leq& C\left(1+\gamma^{1+\frac{1}{c}-\frac{1-\beta}{2}}\right)\left(1+\gamma^{\frac{3}{2}+\frac{1}{c}-\frac{1}{\kappa}}\right) \left(C_0\|(u_0,b_0,a_0)\|_{X^{m,c}(\mr^2)}\right)^2\tau^{\frac{1}{\kappa}-1+\frac{1-\beta}{2}-\frac{1}{c}},
	\end{align*}where in the last inequality   Young's inequality to $\left(1+\gamma^{\frac{m}{2m+2}}\right)^{1-\theta_1}$ has been applied.
	Inserting the above estimate into \eqref{high-3} shows that
	\begin{equation}\label{h-3}
		\begin{split}
		H_2&\leq C\left(1+\gamma^{1+\frac{1}{c}-\frac{1-\beta}{2}}\right)\left(1+\gamma^{\frac{3}{2}+\frac{1}{c}-\frac{1}{\kappa}}\right)\\
		&~~~~~~ \left(C_0\|(u_0,b_0,a_0)\|_{X^{m,c}(\mr^2)}\right)^2 \int_{\frac{t}{2}}^t(t-\tau)^{-\frac{1}{\kappa}}\tau^{\frac{1}{\kappa}-1+\frac{1-\beta}{2}-\frac{1}{c}}\dd\tau\\
		&\leq C\left(1+\gamma^{1+\frac{1}{c}-\frac{1-\beta}{2}}\right)\left(1+\gamma^{\frac{3}{2}+\frac{1}{c}-\frac{1}{\kappa}}\right) \left(C_0\|(u_0,b_0,a_0)\|_{X^{m,c}(\mr^2)}\right)^2 t^{\frac{1-\beta}{2}-\frac{1}{c}}\\
		&\leq C\left(1+\gamma^{1+\frac{1}{c}-\frac{1-\beta}{2}}\right)\left(1+\gamma^{\frac{3}{2}+\frac{1}{c}-\frac{1}{\kappa}}\right) \left(C_0\|(u_0,b_0,a_0)\|_{X^{m,c}(\mr^2)}\right)^2 (1+t)^{\frac{1-\beta}{2}-\frac{1}{c}}.
	\end{split}
	\end{equation}
	
	Collecting the estimates of \eqref{h-0}, \eqref{h-1}, \eqref{h-2}, \eqref{h-3} and using Young's inequality lead to
	\begin{equation}\label{u-f}
		\begin{split}
		&\|\Lambda^\beta u\|_{L^2}
		\\ \leq& c_0\|(u_0,b_0,a_0)\|_{X^{m,c}(\mr^2)} (1+t)^{\frac{1-\beta}{2}-\frac{1}{c}}+C\left(1+\gamma^{1+\frac{1}{c}-\frac{1-\beta}{2}}\right)\\
		&~~~~~~~~~~~~\left(1+\gamma^{\frac{3}{2}+\frac{1}{c}}+\gamma^{\frac{1}{2}+\frac{1}{c}+\frac{\beta}{2}}\right)  \left(C_0\|(u_0,b_0,a_0)\|_{X^{m,c}(\mr^2)}\right)^2 (1+t)^{\frac{1-\beta}{2}-\frac{1}{c}}.
	\end{split}
	\end{equation}
	
The estimates for the $b$ equation are more delicate.	First, we act on the integral representation of $b$ as given in \eqref{b integral} of Proposition \ref{integral}, and   infer that
	\begin{equation}\label{h-b-1}
		\begin{split}
		\|\Lambda^\beta b\|_{L^2}\leq& \left\|\left(K_0+\frac{1}{2}K_1\right)\Lambda^\beta b_0\right\|_{L^2}+\gamma\left\|K_1\Lambda^\beta a_0\right\|_{L^2}\\
		&~~~+\left\|\int_{0}^{t}K_1(t-\tau)\Lambda^\beta \left(b\cdot\nabla u-u\cdot\nabla b\right)\dd \tau\right\|_{L^2}\\
		\stackrel{\mathrm{def}}{=}&L+M+N
		\end{split}
\end{equation}

	According to Lemma $\ref{kernel prop}$, we know that the decay caused by $K_1$ and $K_2$ varies at different frequencies. Therefore, we  divide $L$ into two parts:
	\begin{equation*}
		\begin{split}
	L&\leq	\left\|\w{\left(K_0+\frac{1}{2}K_1\right)}\Lambda^\beta b_0\right\|_{L^2(S_1)}+	\left\|\w{\left(K_0+\frac{1}{2}K_1\right)}\Lambda^\beta b_0\right\|_{L^2(S_2)}\\
		&\stackrel{\mathrm{def}}{=}	L_1+L_2.
	\end{split}
\end{equation*}
	By virtue of Lemma \ref{young+KDf}, \eqref{fren-1} and \eqref{basic},
	it follows that
\begin{equation*}
	\begin{split}
		L_1
		&\leq  C\e^{-\frac{1}{8\gamma}t}\|b_0\|_{H^\beta(S_1)}\\\label{4.7}
		&\leq C\gamma^{-\frac{1-\beta}{2}+\frac{1}{c}}t^{\frac{1-\beta}{2}-\frac{1}{c}}\|b_0\|_{H^m(\mathbb{R}^2)}\\&\leq C\gamma^{-\frac{1-\beta}{2}+\frac{1}{c}}(1+t)^{\frac{1-\beta}{2}-\frac{1}{c}}\|b_0\|_{H^m(\mathbb{R}^2)}.
		\end{split}
\end{equation*}
	Employing  Lemma \ref{young+KDf} and \eqref{fren-3} yields
	\begin{equation*}
		\begin{split}
		L_2
		&\leq C\left\|\e^{-|\xi|^2t}\w{\Lambda^\beta b_0}(\xi)\right\|_{L^2(S_2)}\\
		&\leq C\||\xi|^{\beta} \e^{-|\xi|^2t}\|_{L^{\frac{2c}{2-c}}(\mr^2)}\|\w{b_0}(\xi)\|_{L^{\frac{c}{c-1}}(\mr^2)} \\
		&\leq Ct^{-\frac{\beta}{2}+\frac{c-2}{2c}} \|b_0\|_{L^c(\mr^2)} \\&\leq C(1+t)^{\frac{1-\beta}{2}-\frac{1}{c}} \|b_0\|_{L^c(\mr^2)}.
	\end{split}
\end{equation*}
	Combining the above estimates, we obtain
	\begin{align}\label{h-b-2}	L	&\leq c_1 \left(1+\gamma^{-\frac{1-\beta}{2}+\frac{1}{c}}\right)\|(u_0,b_0,a_0)\|_{X^{m,c}(\mr^2)}(1+t)^{\frac{1-\beta}{2}-\frac{1}{c}} .
	\end{align}

	Similarly, one can estimate the term involving $M$ as
	\begin{equation}\label{h-b-3}
		\begin{split}
		M=\gamma\left\|K_1\Lambda^\beta a_0\right\|_{L^2}&\leq C\gamma(1+\gamma^{-\frac{1-\beta}{2}+\frac{1}{c}})(1+t)^{\frac{1-\beta}{2}-\frac{1}{c}}\|a_0\|_{H^m \cap L^c}\\
		&\leq c_2\left(1+\gamma^{1-\frac{1-\beta}{2}+\frac{1}{c}}\right)\|(u_0,b_0,a_0)\|_{X^{m,c}(\mr^2)}(1+t)^{\frac{1-\beta}{2}-\frac{1}{c}}.
	\end{split}
\end{equation}

	Now,  we do estimate for $N.$
	\begin{align*}
		N&\leq\left\|\int_{0}^{t}K_1(t-\tau)\Lambda^\beta\nabla\cdot\left(u\otimes b\right)\dd \tau\right\|_{L^2}\\
		&\leq \int_{0}^{t} \left\|\w{K_1(t-\tau)\Lambda^\beta\nabla\cdot\left(u\otimes b\right)}\right\|_{L^2(S_1)}\dd \tau+ \int_{0}^{t} \left\|\w{K_1(t-\tau)\Lambda^\beta\nabla\cdot\left(u\otimes b\right)}\right\|_{L^2(S_2)}\dd \tau\\
		&\stackrel{\mathrm{def}}{=}N_1+N_2.
	\end{align*}
For the first part, 	after using  \eqref{fren-2}, it follows that
	\begin{equation*}
		\begin{split}
		N_1&\leq C \int_0^t \left\|\gamma^{-\frac{1}{2}}|\xi|^{-1}\e^{-\frac{1}{8\gamma}(t-\tau)}\w{\Lambda^\beta\nabla\cdot\left(u\otimes b\right)}(\xi)\right\|_{L^2(S_1)}\dd \tau\\
		&\leq C\gamma^{-\frac{1}{2}}\int_0^t\e^{-\frac{1}{8\gamma}(t-\tau)} \left\|\Lambda^\beta(u(\tau)\otimes b(\tau))\right\|_{L^2}\dd\tau.
	\end{split}
\end{equation*}
	Applying \eqref{mul} and Lemma \ref{gag} yields
	\begin{equation*}
	\begin{split}
		&\left\|\Lambda^\beta(u(\tau)\otimes b(\tau))\right\|_{L^2} \\
		\leq& C\|\Lambda^\beta u(\tau)\|_{L^2}\|b(\tau)\|_{L^\infty}
		+C\|\Lambda^\beta b(\tau)\|_{L^{2}}\|u(\tau)\|_{L^{\infty}}\\ \leq& C\left(\|\Lambda^\beta u(\tau),\Lambda^\beta b(\tau)\|_{L^2}\|\nabla u(\tau),\nabla b(\tau)\|_{H^{m}}^{\frac{1}{1+m}}\|u(\tau),b(\tau)\|_{L^2}^\frac{m}{m+1}\right)
		\\ \leq& C\left(1+\gamma^{1+\frac{1}{c}-\frac{1-\beta}{2}}\right)\left(C_0\|u_0,b_0,a_0\|_{X^{m,c}(\R^2)}\right)^{1+\frac{m}{1+m}}\\&~~~~~~~~~~~~~~~~~~~~~~~~~~~~\times \|\nabla u(\tau),\nabla b(\tau)\|_{H^{m}}^{\frac{1}{1+m}}(1+\tau)^{\frac{1-\beta}{2}-\frac{1}{c}},	\end{split}
\end{equation*}
where   \eqref{high-1}  has been used in 	the last inequality. Now, we proceed with the estimates, 
\begin{align*}
	& \int_0^{t}e^{-\frac{1}{8\gamma}(t-\tau)}(1+\tau)^{\frac{1-\beta}{2}-\frac{1}{c}}\|\nabla u(\tau),\nabla b(\tau)\|_{H^{m}(\mathbb{R}^2)}^{\frac{1}{1+m}}d \tau\\
	\leq&  \left(\int_0^t\|\nabla u(\tau),\nabla b(\tau)\|_{H^{m}(\mathbb{R}^2)}^2d \tau\right)^{\frac{1}{2+2m}}\\&~~~~~~~~~~~~~~~~~~\times\left(\int_{0}^{t}\left(e^{-\frac{1}{8\gamma}(t-\tau)}(1+\tau)^{\frac{1-\beta}{2}-\frac{1}{c}}\right)^\frac{2+2m}{1+2m}d \tau\right)^{\frac{1+2m}{2+2m}}\\&\leq C(1+\gamma^{\frac{m}{2m+2}})^{\frac{1}{m+1}}\|(u_0,b_0,a_0)\|_{X^{m}(\R^2)}^{\frac{1}{m+1}}\\&~~~~~~~~~~~~~~~~~~~\times\left(\int_{0}^{t}\left(e^{-\frac{1}{8\gamma}(t-\tau)}(1+\tau)^{\frac{1-\beta}{2}-\frac{1}{c}}\right)^\frac{2+2m}{1+2m}d \tau\right)^{\frac{1+2m}{2+2m}}.
\end{align*}
	By 
	\eqref{p-2}, we have
	\begin{align}\nonumber
		&\int_{0}^{t}\left(\e^{-\frac{1}{8\gamma}(t-\tau)}(1+\tau)^{\frac{1-\beta}{2}-\frac{1}{c}}\right)^\frac{2+2m}{1+2m}\dd \tau \\\label{hm-exp-integral}
		&\leq \begin{cases}
			C(1+t)^{-1}\left(\gamma +\gamma^2\right), &(\frac{1}{c}-\frac{1}{2}+\frac{\beta}{2})\frac{2+2m}{1+2m}=1,\\
			C(1+t)^{(\frac{1-\beta}{2}-\frac{1}{c})\frac{2+2m}{1+2m}}	\gamma , &(\frac{1}{c}-\frac{1}{2}+\frac{\beta}{2})\frac{2+2m}{1+2m}<1,\\
			C(1+t)^{(\frac{1-\beta}{2}-\frac{1}{c})\frac{2+2m}{1+2m}}\left(\gamma +\gamma^{(\frac{1}{c}-\frac{1}{2}+\frac{\beta}{2})\frac{2+2m}{1+2m}}\right),&(\frac{1}{c}-\frac{1}{2}+\frac{\beta}{2})\frac{2+2m}{1+2m}>1.
		\end{cases}
	\end{align}
	We note that the lowest power of $\gamma$ appearing in the above expression is 1.  Collecting the above results, it
	yields that
\begin{align*}\nonumber
	&N_1\leq  C\left(C_0\|(u_0,b_0,a_0)\|_{X^{m,c}(\R^2)}\right)^2\left(1+\gamma^{1+\frac{1}{c}-\frac{1-\beta}{2}}\right)\phi(\gamma)(1+t)^{\frac{1-\beta}{2}-\frac{1}{c}},
\end{align*}
where \begin{align}\nonumber
	\phi(\gamma)&=\begin{cases}
		\left(\gamma^{\frac{m}{2+2m}} +\gamma^{\frac{2+3m}{2+2m}}\right), &(\frac{1}{c}-\frac{1}{2}+\frac{\beta}{2})\frac{2+2m}{1+2m}=1,\\
		\gamma^{\frac{m}{2+2m}} , &(\frac{1}{c}-\frac{1}{2}+\frac{\beta}{2})\frac{2+2m}{1+2m}<1,\\
		\left(\gamma^{\frac{m}{2+2m}} +\gamma^{\frac{1}{c}-1+\frac{\beta}{2}}\right),&(\frac{1}{c}-\frac{1}{2}+\frac{\beta}{2})\frac{2+2m}{1+2m}>1.
	\end{cases}\\\label{phi-estimate}
	&\leq C\left(1+\gamma^{\frac{3}{2}+\frac{1}{c}}+\gamma^{\frac{1}{2}+\frac{1}{c}+\frac{\beta}{2}}\right).
\end{align}
 By Lemma \ref{kernel prop},
	\begin{align*}
		N_2&\leq C\int_{0}^{\frac{t}{2}} \left\|\e^{-|\xi|^2(t-\tau)}\w{\Lambda^\beta\nabla\cdot\left(u\otimes b\right)}\right\|_{L^2(S_2)}\dd \tau+\int_{\frac{t}{2}}^{t} \left\|\e^{-|\xi|^2(t-\tau)}\w{\Lambda^\beta\nabla\cdot\left(u\otimes b\right)}\right\|_{L^2(S_2)}\dd \tau\\
		&\stackrel{\mathrm{def}}{=}	N_{21}+N_{22}.
	\end{align*}
	At this point, the estimate of $N_{21}$ and $N_{22}$ can be treated similarly to that of $H_1$ and $H_2$. In fact, it follows that
	\begin{align*}
	N_{2} &	\leq C\left(1+\gamma^{1+\frac{1}{c}-\frac{1-\beta}{2}}\right)\left(1+\gamma^{\frac{3}{2}+\frac{1}{c}}+\gamma^{\frac{1}{2}+\frac{1}{c}+\frac{\beta}{2}}\right)  \\
		&~~~~~~~~~~~~\times\left(C_0\|(u_0,b_0,a_0)\|_{X^{m,c}(\mr^2)}\right)^2 (1+t)^{\frac{1-\beta}{2}-\frac{1}{c}}.
	\end{align*}
By collecting the estimate of $N_{1}$ and $N_{2}$ together, we get
	\begin{equation}\label{h-b}
		\begin{split}
		N	&\leq C\left(1+\gamma^{1+\frac{1}{c}-\frac{1-\beta}{2}}\right)\left(1+\gamma^{2+\frac{1}{c}}+\gamma^{\frac{1}{2}+\frac{1}{c}+\frac{\beta}{2}}\right)  \left(C_0\|(u_0,b_0,a_0)\|_{X^{m,c}(\mr^2)}\right)^2 (1+t)^{\frac{1-\beta}{2}-\frac{1}{c}}.
	\end{split}
	\end{equation}
	Inserting \eqref{h-b-2}, \eqref{h-b-3} and \eqref{h-b} into \eqref{h-b-1} yields
	\begin{equation}\label{h-b-f}
	\begin{split}
		&\|\Lambda^\beta b(t)\|_{L^2}\\
		\leq& (c_1+c_2) \left(1+\gamma^{1-\frac{1-\beta}{2}+\frac{1}{c}}\right)\|(u_0,b_0,a_0)\|_{X^{m,c}(\mr^2)}(1+t)^{\frac{1-\beta}{2}-\frac{1}{c}}\\
		&+C\left(1+\gamma^{1+\frac{1}{c}-\frac{1-\beta}{2}}\right)\left(1+\gamma^{2+\frac{1}{c}}+\gamma^{\frac{1}{2}+\frac{1}{c}+\frac{\beta}{2}}\right)  \left(C_0\|(u_0,b_0,a_0)\|_{X^{m,c}(\mr^2)}\right)^2 (1+t)^{\frac{1-\beta}{2}-\frac{1}{c}},
	\end{split}
\end{equation}where  Young's inequality has been used to the polynomial of $\gamma$, which ensures us to bound it in terms of the lowest and highest powers of $\gamma$.

At last, by adding \eqref{u-f} and \eqref{h-b-f} together, one can get
	\begin{align*}
		&\|\Lambda^\beta u(t)\|_{L^2}+	\|\Lambda^\beta b(t)\|_{L^2}\\
		\leq &(c_0+c_1+c_2) \left(1+\gamma^{1-\frac{1-\beta}{2}+\frac{1}{c}}\right)\|(u_0,b_0,a_0)\|_{X^{m,c}(\mr^2)}(1+t)^{\frac{1-\beta}{2}-\frac{1}{c}}\\
		&~~+C\left(1+\gamma^{1+\frac{1}{c}-\frac{1-\beta}{2}}\right)\left(1+\gamma^{2+\frac{1}{c}}+\gamma^{\frac{1}{2}+\frac{1}{c}+\frac{m}{2}}\right)  \left(C_0\|(u_0,b_0,a_0)\|_{X^{m,c}(\mr^2)}\right)^2 (1+t)^{\frac{1-\beta}{2}-\frac{1}{c}},
	\end{align*}
	which implies \eqref{high-2} by selecting $C_0\geq 4\left(c_0+c_1+c_2\right)$ and a sufficiently small initial value  such that \begin{align*}
		CC_0\left\| (u_0,b_0,a_0)\right\|_{X^{m,c}(\mr^2)}\left(1+\gamma^{2+\frac{1}{c}}+\gamma^{\frac{1}{2}+\frac{1}{c}+\frac{m}{2}}\right) \leq \frac{1}{4}.
	\end{align*}
	Here, we assume that $\left\|(u_0,b_0,a_0)\right\|_{X^m(\mr^2)}\leq \varepsilon$, and \begin{align*}
		\varepsilon \leq  \dfrac{1}{4CC_0\left(1+\gamma^{2+\frac{1}{c}}+\gamma^{\frac{1}{2}+\frac{1}{c}+\frac{m}{2}}\right) }.
	\end{align*}
	
	{\bf Step 2.} 	We are ready to prove \eqref{th2-2}. Using a similar approach as in Step 1, we make the ansatz that, for $t<T$,
	\begin{equation}\label{5.1}
		\begin{split}
		&\|\Lambda^\varrho b(t)\|_{L^2}\leq C_0\left(1+\gamma^{1+\frac{1}{c}-\frac{1-\varrho}{2}}\right)\|(u_0,b_0,a_0)\|_{X^{m,c}(\mr^2)} (1+t)^{\frac{1-\varrho}{2}-\frac{1}{c}}.
\end{split}	\end{equation}
	Here $C_0\geq C> 0$  is a constant that will be specified later, where $C$ is the constant appearing on the right-hand side of \eqref{th1-1} in Theorem \ref{thm1}, independent of $\gamma$ and $0\leq \varrho<m+1$. 
	By the continuity argument (\cite{tao}), it remains to verify
	\begin{align}\label{5.2}
		&\|\Lambda^\varrho b(t)\|_{L^2}\leq \frac{C_0}{2}\left(1+\gamma^{1+\frac{1}{c}-\frac{1-\varrho}{2}}\right)\|(u_0,b_0,a_0)\|_{X^{m,c}(\mr^2)} (1+t)^{\frac{1-\varrho}{2}-\frac{1}{c}}.
	\end{align}
	
	Since \eqref{th2-1} with $0\leq \varrho\leq m$ has been proved  in Step 1,  we only need to consider the case when $m< \varrho<m+1$ in order to prove \eqref{th2-2}.
	
	Similar to \eqref{h-b-1},	for  $m<\varrho<m+1$,  we have
	\begin{equation}\label{bm+1}
		\begin{split}
		\|\Lambda^\varrho b(t)\|_{L^2}\leq& \left\|\left(K_0+\frac{1}{2}K_1\right)\Lambda^\varrho b_0\right\|_{L^2}+\gamma\left\|K_1\Lambda^\varrho a_0\right\|_{L^2}\\
		&+\left\|\int_{0}^{t}K_1(t-\tau)\widehat{\Lambda^\varrho \left(b\cdot\nabla u-u\cdot\nabla b\right)}\dd \tau\right\|_{L^2(S_1)}\\
		&+\left\|\int_{0}^tK_1(t-\tau)\widehat{\Lambda^\varrho \left(b\cdot\nabla u-u\cdot\nabla b\right)\dd \tau}\right\|_{L^2(S_2)}\\
		\stackrel{\mathrm{def}}{=}&O+P+Q_1+Q_2.
	\end{split}
	\end{equation}
 Using similar techniques to get \eqref{h-b-2}, we obtain
	\begin{align}\label{bm+1-1}
		O&\leq c_3\left(1+\gamma^{\frac{1}{c}-\frac{1-\varrho}{2}}\right)\|(u_0,b_0,a_0)\|_{X^{m,c}(\mr^2)}  (1+t)^{\frac{1-\varrho}{2}-\frac{1}{c}}.
	\end{align}
	For $P$,
	\begin{align*}
		P=\gamma\left\|K_1\Lambda^\varrho a_0\right\|_{L^2}&\leq \gamma \left\|\w{K_1\Lambda^\varrho a_0}\right\|_{L^2(S_1)}+\gamma \left\|\w{K_1\Lambda^\varrho a_0}\right\|_{L^2(S_2)}\\
		&\define 	P_1+P_2.
	\end{align*}
	Taking  $\vartheta=\varrho-m$ in \eqref{fren-2} and  uisng \eqref{basic} yield
	\begin{align*}
		P_1&\leq C \gamma \left\| \gamma^{-\frac{\varrho-m}{2}}|\xi|^{-(\varrho-m)}\e^{-\frac{1}{8\gamma}t} |\xi|^{\varrho}\w{a_0}\right\|_{L^2}\\
		&\leq C\gamma^{1-\frac{\varrho-m}{2}}\e^{-\frac{1}{8\gamma}t}\|\Lambda^{m}a_0\|_{L^2}\\
		&\leq C\gamma^{\frac{1}{2}+\frac{1}{c}+\frac{m}{2}}t^{\frac{1-\varrho}{2}-\frac{1}{c}}\left\|a_0\right\|_{H^m}.
	\end{align*}
For $P_2$, it follows that
	\begin{align*}
		P_2&\leq C\gamma \left\| \e^{-|\xi|^2t} |\xi|^\varrho \w{a_0}\right\|_{L^2(S_2)}\\
		&\leq C\gamma \left\||\xi|^{\varrho}\e^{-|\xi|^2t}\right\|_{L^{\frac{2c}{2-c}}(\mr^2)}\left\|\w{a_0}\right\|_{L^{\frac{c}{c-1}}(\mr^2)}\\
		&\leq C\gamma t^{-\frac{\varrho}{2}+\frac{c-2}{2c}}\|a_0\|_{L^c}.
	\end{align*}
	Then,
	\begin{align}\label{bm+1-2}
		P\leq c_4 \left(\gamma+\gamma^{\frac{1}{2}+\frac{1}{c}+\frac{m}{2}}\right) \|(u_0,b_0,a_0)\|_{X^{m,c}(\mr^2)} (1+t)^{\frac{1-\varrho}{2}-\frac{1}{c}}.
	\end{align}
	
As	for $Q_1$, using \eqref{fren-2} with $\vartheta=1$ firstly, we get
	\begin{equation*}\begin{split}
		Q_1
		&\leq C\gamma^{-\frac{1}{2}}\int_0^t\e^{-\frac{1}{8\gamma}(t-\tau)} \left\|\Lambda^\varrho(u(\tau)\otimes b(\tau))\right\|_{L^2}\dd\tau \\
		&\leq C\gamma^{-\frac{1}{2}}\int_0^{t}\e^{-\frac{1}{8\gamma}(t-\tau)}\left(\|\Lambda^\varrho u(\tau)\|_{L^{i_1}}\|b(\tau)\|_{L^{i_2}}+\|\Lambda^\varrho b(\tau)\|_{L^{2}}\|u(\tau)\|_{L^{\infty}}\right)\dd \tau\\&\define  Q_{11}+Q_{12},
\end{split}	\end{equation*}where 	$\frac{1}{i_1}+\frac{1}{i_2}=\frac{1}{2}.$
$Q_{12}$ can be estimated in a very similar way as $N_1$.  Indeed,  we obtain
\begin{align*}\nonumber
	&Q_{12}\leq  C\left(C_0\|(u_0,b_0,a_0)\|_{X^{m,c}(\R^2)}\right)^2\left(1+\gamma^{1+\frac{1}{c}-\frac{1-\varrho}{2}}\right)\phi_1(\gamma)(1+t)^{\frac{1-\varrho}{2}-\frac{1}{c}},
\end{align*}where $\phi_1(\gamma)$ is a polynomial function with nonnegative exponents of $\gamma.$

Now, we do estimate for $Q_{12}.$
Using the Gagliardo-Nirenberg interpolation inequality and \eqref{th2-1} with $\varrho=0$, we get
\begin{equation*}\begin{split}
	&\|\Lambda^\varrho u(\tau)\|_{L^{i_1}}\\ \leq& C\left\| \Lambda^{m+1}u(\tau)\right\|_{L^2}^{\theta_1}\left\| u(\tau)\right\|_{L^2}^{1-\theta_1}\\ \leq&C\left\| \Lambda^{m+1}u(\tau)\right\|_{L^2}^{\theta_1}\left(C_0\|(u_0,b_0,a_0)\|_{X^{m,c}(\mr^2)}  \right)^{1-\theta_1}\\ &~~~~~~~~~~~~~~~~~~~~~~~~\times\left(1+\gamma^{\frac{1}{2}+\frac{1}{c}}\right)^{1-\theta_1}(1+\tau)^{\left(\frac{1}{2}-\frac{1}{c}\right)\left(1-\theta_1\right)},
\end{split}
\end{equation*}
where $\varrho-\frac{2}{i_1}=m\theta_1-(1-\theta_1).$
Since $1-\frac{2}{i_2}<m+1,$	by the assumption \eqref{5.1} and \eqref{th2-1}, we have
\begin{align*}
	\|b(\tau)\|_{L^{i_2}}\leq  \|\Lambda^{1-\frac{2}{i_2}}b (\tau)\|_{L^2}\leq C_0\left(1+\gamma^{1+\frac{1}{c}-\frac{1}{i_2}}\right)\|(u_0,b_0,a_0)\|_{X^{m,c}(\mr^2)}\left(1+\tau\right)^{\frac{1}{i_2}-\frac{1}{c}}.
\end{align*}
It follows that $$\frac{1}{i_2}-\frac{1}{c}+\left(\frac{1}{2}-\frac{1}{c}\right)\left(1-\theta_1\right)=\frac{1-\varrho}{2}-\frac{1}{c}$$ by choosing $\theta_1=\dfrac{\frac{1}{c}}{\frac{m}{2}+\frac{1}{c}}$.
Therefore, 	\begin{equation*}\begin{split}
		&\|\Lambda^\varrho u(\tau)\|_{L^{i_1}}\|b(\tau)\|_{L^{i_2}}\\ \leq&
	 C\left(1+\gamma^{1+\frac{1}{c}-\frac{1}{i_2}}\right)\left(1+\gamma^{\frac{1}{2}+\frac{1}{c}}\right)^{1-\theta_1}\left(C_0\|(u_0,b_0,a_0)\|_{X^{m,c}(\mr^2)}  \right)^{2-\theta_1} \\
		&~~~~~~ ~~~~~~~~~~\times(1+\tau)^{\frac{1-\varrho}{2}-\frac{1}{c}} \left\| \Lambda^{m+1}u\right\|_{L^2}^{\theta_1}.
\end{split}	\end{equation*}
By  plugging this estimate  into $Q_{11}$ and using a similar way to estimate $N_1,$ we are able to get 
\begin{align*}\nonumber
	&Q_{11}\leq  C\left(C_0\|(u_0,b_0,a_0)\|_{X^{m,c}(\R^2)}\right)^2\left(1+\gamma^{1+\frac{1}{c}-\frac{1-\varrho}{2}}\right)\phi_2(\gamma)(1+t)^{\frac{1-\varrho}{2}-\frac{1}{c}},
\end{align*}where $\phi_2(\gamma)$ is a polynomial function with nonnegative exponents of $\gamma.$
To conclude, \begin{align}\label{bm+1-3}
	&Q_{1}\leq  C\left(C_0\|(u_0,b_0,a_0)\|_{X^{m,c}(\R^2)}\right)^2\left(1+\gamma^{1+\frac{1}{c}-\frac{1-\varrho}{2}}\right)\left(\phi_1(\gamma)+\phi_2(\gamma)\right)(1+t)^{\frac{1-\varrho}{2}-\frac{1}{c}}.
\end{align}	
Similarly \eqref{phi-estimate}, we have
$$\phi_1(\gamma)+\phi_2(\gamma)\leq C\left(1+\gamma^{3+\frac{2}{c}}+\gamma^{2+\frac{3}{c}+\frac{m}{2}}\right).$$
	
	For $Q_2$,
	\begin{equation*}
		\begin{split}
	Q_2&\leq \int_{0}^t \left\|\e^{-|\xi|^2(t-\tau)}\w{\Lambda^\varrho \left(b\cdot\nabla u-u\cdot\nabla b\right)}\right\|_{L^2(S_2)}\dd\tau\\
		&\leq \int_{0}^\frac{t}{2} \left\|\e^{-|\xi|^2(t-\tau)}\w{\Lambda^\varrho \left(b\cdot\nabla u-u\cdot\nabla b\right)}\right\|_{L^2(S_2)}\dd\tau\\
		&~~~~~~~~+\int_{\frac{t}{2}}^t \left\|\e^{-|\xi|^2(t-\tau)}\w{\Lambda^\varrho \left(b\cdot\nabla u-u\cdot\nabla b\right)}\right\|_{L^2(S_2)}\dd\tau\\
		&\stackrel{\mathrm{def}}{=}Q_{21}+Q_{22}.
	\end{split}
	\end{equation*}
	Applying Lemma \ref{heatkernel} yields 
	\begin{align*}
		Q_{21}&= \int_{0}^\frac{t}{2} \left\|\e^{-|\xi|^2(t-\tau)}\w{\Lambda^\varrho \left(b\cdot\nabla u-u\cdot\nabla b\right)}\right\|_{L^2(S_2)}\dd\tau\\
		&\leq C \int_0^\frac{t}{2}(t-\tau)^{-1-\frac{\varrho}{2}}\left\| u(\tau)\otimes b(\tau)\right\|_{L^1}\dd\tau\\
		&\leq C t^{-1-\frac{\varrho}{2}}\int_0^\frac{t}{2} \|u(\tau),b(\tau)\|_{L^2}^2\dd \tau.
	\end{align*}
	By analyzing the integrand, we get
	\begin{align*}
		&\|u(\tau),b(\tau)\|_{L^2}^2\\
		\leq &\|u(\tau),b(\tau)\|_{L^2}\|u(\tau),b(\tau)\|_{H^m}\\
		\leq & \left(C_0\|(u_0,b_0,a_0)\|_{X^{m,c}(\mr^2)}\right)^2\left(1+\gamma^{\frac{1}{2}+\frac{1}{c}}\right)\left(1+\gamma^{\frac{m}{2m+2}}\right)\tau^{\frac{1}{2}-\frac{1}{c}},
	\end{align*}
	where \eqref{th2-1} and \eqref{th1-1} have been used in the last inequality.
Hence, 
	\begin{align*}
		Q_{21}	&\leq Ct^{-1-\frac{\varrho}{2}}\left(C_0\|(u_0,b_0,a_0)\|_{X^{m,c}(\mr^2)}\right)^2\left(1+\gamma^{\frac{1}{2}+\frac{1}{c}}\right)\left(1+\gamma^{\frac{m}{2m+2}}\right)\int_{0}^{\frac{t}{2}}\tau^{\frac{1}{2}-\frac{1}{c}}\dd\tau\\
		&\leq C\left(1+\gamma^{\frac12+\frac{1}{c}+\frac{m}{2m+2}}\right)\left(C_0\|(u_0,b_0,a_0)\|_{X^{m,c}(\mr^2)}\right)^2 t^{\frac{1-\varrho}{2}-\frac{1}{c}}\\
		&\leq C\left(1+\gamma^{1+\frac{1}{c}-\frac{1-\varrho}{2}}\right)\left(C_0\|(u_0,b_0,a_0)\|_{X^{m,c}(\mr^2)}\right)^2 (1+t)^{\frac{1-\varrho}{2}-\frac{1}{c}}.
	\end{align*}

	Comparatively speaking, the estimation of $Q_{22}$ is more challenging. To begin with, by  utilizing Lemma \ref{young+KDf} and \eqref{mul}, we have
	\begin{align*}
		Q_{22}&=\int_{\frac{t}{2}}^t \left\|\e^{-|\xi|^2(t-\tau)}\w{\Lambda^\varrho \left(b\cdot\nabla u-u\cdot\nabla b\right)}\right\|_{L^2(S_2)}\dd\tau\\
		&\leq \int_{\frac{t}{2}}^t (t-\tau)^{-\frac{\varrho-m}{2}-\frac{1}{k}} \left\|\Lambda^{m-1} \left(b\cdot\nabla u-u\cdot\nabla b\right)\right\|_{L^k(\mr^2)}\dd\tau\\
		&\leq \int_{\frac{t}{2}}^t (t-\tau)^{-\frac{\varrho-m}{2}-\frac{1}{k}} \left(\left\|\Lambda^{m} u(\tau)\right\|_{L^2}\|b(\tau)\|_{L^\frac{2k}{2-k}}+\left\|\Lambda^{m} b(\tau)\right\|_{L^{r_1}}\|u(\tau)\|_{L^{r_2}}\right)\dd\tau,
	\end{align*}
	where 
	\begin{equation*}
		1-\frac{\varrho-m}{2}>\frac{1}{k}>\frac{1}{2},~~\frac{1}{r_1}+\frac{1}{r_2}=\frac{1}{k}.
	\end{equation*} 
The upper bound $1-\frac{\varrho-m}{2}$ of $\frac{1}{k}$ is selected to ensure  $$\int_{\frac{t}{2}}^t (t-\tau)^{-\frac{\varrho-m}{2}-\frac{1}{k}} \dd\tau<\infty.$$
Since $\frac{2k-2}{k} <1<m+1,$  by the assumption \eqref{5.1} or \eqref{th2-1}, we have
\begin{equation*}\begin{split}
\|b(\tau)\|_{L^\frac{2k}{2-k}}\leq&C\|\Lambda^\frac{2k-2}{k} b(\tau)\|_{L^2}\\ \leq&CC_0\|(u_0,b_0,a_0)\|_{X^{m,c}(\mr^2)}\left(1+\gamma^{\frac{3}{2}+\frac{1}{c}-\frac{1}{k}}\right)(1+\tau)^{\frac1k-\frac12-\frac1c}.
\end{split}\end{equation*}
As a result, by using \eqref{th2-1} with $\beta=m$, we get
	\begin{equation*}\begin{split}
		&\left\|\Lambda^{m} u(\tau)\right\|_{L^2}\|b(\tau)\|_{L^\frac{2k}{2-k}}\\ \leq&C\left(C_0\|(u_0,b_0,a_0)\|_{X^{m,c}(\mr^2)}\right)^2 \left(1+\gamma^{\frac{1}{2}+\frac{1}{c}+\frac{m}{2}}\right)\left(1+\gamma^{\frac{3}{2}+\frac{1}{c}-\frac{1}{k}}\right) (1+\tau)^{\frac{1-m}{2}-\frac{1}{c}+\frac{1}{k}-\frac12-\frac{1}{c}}\\ \leq&C\left(C_0\|(u_0,b_0,a_0)\|_{X^{m,c}(\mr^2)}\right)^2 \left(1+\gamma^{1+\frac{1}{c}-\frac{1-\varrho}{2}}\right)\left(1+\gamma^{\frac{3}{2}+\frac{1}{c}-\frac{1}{k}}\right) (1+\tau)^{\frac{1-m}{2}-\frac{1}{c}+\frac{1}{k}-1}.\end{split}
	\end{equation*}
 When the $m$-order  derivative is evaluated at $b,$ we choose \begin{equation*}
 	1-\frac{\varrho-m}{2}>\frac{1}{k}\geq\frac{1}{2}+\frac{1}{r_2}+\frac{m-\varrho}{2},
 \end{equation*}
 \begin{equation*}\begin{split}
r_2: \begin{cases}
2< r_2<\min\{\frac{2}{1-m},\frac{2}{\varrho-m}\}, &\text{if} ~~m<1,\\[2mm]
 2< r_2<\frac{2}{\varrho-m},&\text{if} ~~m>1.
\end{cases}\end{split}	\end{equation*}
Then, we  can obtain
\begin{equation*}\begin{split}
		&\left\|\Lambda^{m} b(\tau)\right\|_{L^{r_1}}\|u(\tau)\|_{L^{r_2}}\\
		\leq& C\left\|\Lambda^{m+1-\frac{2}{r_1}}b(\tau)\right\|_{L^2}\|\Lambda^{1-\frac{2}{r_2}}u(\tau)\|_{L^{2}}\\
		\leq& C\left(C_0\|(u_0,b_0,a_0)\|_{X^{m,c}(\mr^2)}\right)^2\left(1+\gamma^{\frac{1}{2}+\frac{1}{c}+\frac{m+1-\frac{2}{r_1}}{2}}\right)\\
		&~~~~~~~~~~~~~~~~~~~~\times(1+\tau)^{\frac{-m+\frac{2}{r_1}}{2}-\frac{1}{c}}\left(1+\gamma^{1-\frac{1}{r_2}+\frac{1}{c}}\right)(1+\tau)^{\frac{1}{r_2}-\frac{1}{c}}\\
		\leq& C\left(1+\gamma^{1+\frac{1}{c}-\frac{1-\varrho}{2}}\right)\left(1+\gamma^{1-\frac{1}{r_2}+\frac{1}{c}}\right)\left(C_0\|(u_0,b_0,a_0)\|_{X^{m,c}(\mr^2)}\right)^2(1+\tau)^{\frac{-m}{2}-\frac{1}{c}-\frac{1}{c}+\frac{1}{k}}\\
		\leq& C\left(1+\gamma^{1+\frac{1}{c}-\frac{1-\varrho}{2}}\right)\left(1+\gamma^{1-\frac{1}{r_2}+\frac{1}{c}}\right)\left(C_0\|(u_0,b_0,a_0)\|_{X^{m,c}(\mr^2)}\right)^2(1+\tau)^{\frac{1-m}{2}-\frac{1}{c}-1+\frac{1}{k}}.
	\end{split}	\end{equation*}
By plugging these estimates into $Q_{22},$ we obtain
	\begin{equation*}\begin{split}
		Q_{22}
		&\leq  C\left(1+\gamma^{1+\frac{1}{c}-\frac{1-\varrho}{2}}\right)\left(1+\gamma^{1-\frac{1}{r_2}+\frac{1}{c}}+\gamma^{\frac{3}{2}+\frac{1}{c}-\frac1k}\right)\left(C_0\|(u_0,b_0,a_0)\|_{X^{m,c}(\mr^2)}\right)^2\\
		&~~~~~~~~~~~~~~~~~~~~~~~~~~~~~~~~~~~~~~\times\int_{\frac{t}{2}}^t (t-\tau)^{-\frac{\varrho-m}{2}-\frac{1}{k}} (1+\tau)^{\frac{-1-m}{2}-\frac{1}{c}+\frac{1}{k}}\dd\tau\\
		&\leq  C\left(1+\gamma^{1+\frac{1}{c}-\frac{1-\varrho}{2}}\right)\left(1+\gamma^{1-\frac{1}{r_2}+\frac{1}{c}}+\gamma^{\frac{3}{2}+\frac{1}{c}-\frac1k}\right)\left(C_0\|(u_0,b_0,a_0)\|_{X^{m,c}(\mr^2)}\right)^2\\
		&~~~~~~~~~~~~~~~~~~~~~~~~~~~~~~~~~~~~~~\times (1+t)^{\frac{-1-m}{2}-\frac{1}{c}+\frac{1}{k}} t^{1-\frac{\varrho-m}{2}-\frac{1}{k}}\\
		&\leq C\left(1+\gamma^{1+\frac{1}{c}-\frac{1-\varrho}{2}}\right)\left(1+\gamma^{1-\frac{1}{r_2}+\frac{1}{c}}+\gamma^{\frac{3}{2}+\frac{1}{c}-\frac1k}\right)\\&~~~~~~~~~~~~~~~~\times\left(C_0\|(u_0,b_0,a_0)\|_{X^{m,c}(\mr^2)}\right)^2 (1+t)^{\frac{1-\varrho}{2}-\frac{1}{c}}.
	\end{split}	\end{equation*}
Thus,	collecting the  estimates  of $Q_{21}$ and $Q_{22}$ yields
	\begin{equation}\label{bm+1-4}\begin{split}
		Q_2&=\left\|\int_{0}^tK_1(t-\tau)\Lambda^\varrho \left(b\cdot\nabla u-u\cdot\nabla b\right)\dd \tau\right\|_{L^2(S_2)}\\
		&\leq C\left(1+\gamma^{1+\frac{1}{c}-\frac{1-\varrho}{2}}\right)\left(1+\gamma^{1-\frac{1}{r_2}+\frac{1}{c}}+\gamma^{\frac{3}{2}+\frac{1}{c}-\frac1k}\right)\\&~~~~~~~~~~~~~~~~\times\left(C_0\|(u_0,b_0,a_0)\|_{X^{m,c}(\mr^2)}\right)^2 (1+t)^{\frac{1-\varrho}{2}-\frac{1}{c}}.
\end{split}	\end{equation}
	
Inserting \eqref{bm+1-1}-\eqref{bm+1-4}	into \eqref{bm+1}, and applying Young's inequality to the polynomial in $\gamma$, it yields that
	\begin{align*}
		&\left\| \Lambda^\varrho b(t)\right\|_{L^2}\\
		\leq& \left(c_3+c_4\right)\left(1+\gamma^{1+\frac{1}{c}-\frac{1-\varrho}{2}}\right)\|(u_0,b_0,a_0)\|_{X^{m,c}(\mr^2)} (1+t)^{\frac{1-\varrho}{2}-\frac{1}{c}}\\&+C\left(1+\gamma^{1+\frac{1}{c}-\frac{1-\varrho}{2}}\right)\left(1+\gamma^{3+\frac{2}{c}}+\gamma^{2+\frac{3}{c}+\frac{m}{2}}\right)\left(C_0\|(u_0,b_0,a_0)\|_{X^{m,c}(\mr^2)} \right)^2(1+t)^{\frac{1-\varrho}{2}-\frac{1}{c}}.
	\end{align*}
 By selecting $C_0\geq 4\left(c_3+c_4\right)$ and a sufficiently  small initial data 
	 $$\left\|(u_0,b_0,a_0)\right\|_{X^{m,c}(\mr^2)}\leq \varepsilon,$$ with \begin{align*}
		\varepsilon \leq  \dfrac{1}{4CC_0\left(1+\gamma^{3+\frac{2}{c}}+\gamma^{2+\frac{3}{c}+\frac{m}{2}}\right)},
	\end{align*}
we can derive \eqref{5.2}.

\end{proof}

\vspace{2mm}
{\bf Conflict of interest:} The authors declare that they have no conflict of interest.
\vspace{2mm}

\section*{Acknowledgments}
Yu was partially supported by the	National Natural Science Foundation of China (NNSFC) (No.11901040), Beijing Natural	Science Foundation (BNSF) (No.1204030) and Beijing Municipal Education Commission
(KM202011232020).


\begin{thebibliography}{00}


\bibitem{abidi-2017-GlobalSolution3D}
H. Abidi and P. Zhang, {\em On the global solution of a 3-D MHD
	system with initial data near equilibrium}, Comm. Pure Appl. Math. 70 (2017), no. 8, 1509--1561.
\bibitem{A}R. Agapito and M. Schonbek,  {\em Non-uniform deay for MHD equations with and without magnetic diffusion}, Commun. Partial Differ. Equ. 32 (2007), 1791--1812.


\bibitem{wu-2022-wave}
R. Ji, J. Wu and X. Xu, {\em Global well-posedness of the 2D MHD equations of damped wave type in Sobolev space}, SIAM J. Math. Anal. 54 (2022) 6018--6053. 

\bibitem{jiang-2021-decay}
F. Jiang and S. Jiang, {\em Asymptotic behaviors of global solutions to the two-dimensional non-resistive MHD equations with large initial perturbations}, Adv. Math.  393 (2021), 108084.

\bibitem{xie-2024-zamp}
M. Jin, Q. Jiu and Y. Xie, {\em Global well-posedness and optimal decay for incompressible MHD equations with fractional dissipation and magnetic diffusion},  Z. Angew. Math. Phys.  75 (2024), no. 2, 73.

\bibitem{jiu-2015-MHD}
Q. Jiu, D. Niu, J. Wu, X. Xu and H. Yu, {\em The 2D magnetohydrodynamic equations with magnetic
	diffusion}, Nonlinearity 28 (2015), 3935--3955.

\bibitem{kato-1988-CommutatorEstimatesEuler}
T. Kato and G. Ponce, {\em Commutator estimates and the Euler and Navier-Stokes equations}, Comm. Pure Appl. Math. 41 (1988), no. 7, 891--907.

\bibitem{kenig}
C. Kenig, G. Ponce and L. Vega, {\em Well-posedness of the initial value problem for the Korteweg-de Vries equation},  J. Amer. Math. Soc. 4 (1991), no. 2, 323--347.



	\bibitem{lin-2015-GlobalSmallSolutions}
F. Lin, L. Xu and P. Zhang, {\em Global small solutions of 2-D incompressible MHD system}, J. Differ. Equ. 259 (2015), no. 10, 5440--5485.

\bibitem{lin-2014-GlobalSmallSolutions}
F. Lin and P. Zhang, {\em Global small solutions to an MHD-type system: the three-dimensional case}, Comm. Pure Appl. Math. 67 (2014), no. 4, 531--580.

\bibitem{fourier-2009-nonlinear-ponce-hausdroffyoung}
F. Linares and G. Ponce, {\em Introduction to Nonlinear Dispersive Equations}, Springer New York, 2009. 

\bibitem{wave-2021-jde-fouriersobolev}
T. Matsui, R. Nakasato and T. Ogawa, {\em Singular limit for the magnetohydrodynamics of the damped wave type in the critical Fourier-Sobolev space}, J. Differ. Equ. 271 (2021),  414--446.

\bibitem{miao-2008-heat}
C. Miao, B. Yuan and B. Zhang, {\em Well-posedness of the Cauchy problem for the fractional power dissipative equations}, Nonlinear Anal. 68 (2008), 461--484.

\bibitem{gag-nirenberg}
L. Nirenberg, {\em On elliptic partial differential equations}, Ann. Scuola Norm. Sup. Pisa Cl. Sci. (3) 13 (1959), 115--162.

\bibitem{panGlobalClassicalSolutions2018}
R. Pan, Y. Zhou and Y. Zhu, {\em Global classical solutions of three dimensional viscous MHD system without magnetic diffusion on periodic boxes}, Arch. Ration. Mech. Anal. 227 (2018), no. 2, 637--662.

\bibitem{ren2014global}
X. Ren, J. Wu, Z. Xiang and Z. Zhang, {\em Global existence and decay of
	smooth solution for the 2-D MHD equations without magnetic diffusion},
J. Funct. Anal. 267 (2014), no. 2, 503--541.

\bibitem{schonbek-1996-optimal}
M. Schonbek, T. Schonbek and E. S{\"u}li, {\em Large-time behaviour of solutions to the magnetohydrodynamics equations}, Math. Ann. 304 (1996), 717--756.



\bibitem{sermange-1983}
M. Sermange and R. Temam, {\em Some mathematical questions related to the MHD equations}, Comm. Pure Appl. Math. 36 (1983), no. 5, 635--664.


\bibitem{sun-2023-wave}
W. Sun and W. Wang, {\em Global existence and uniqueness of the 2D damped wave-type MHD equations}, Z. Angew. Math. Phys. 74 (2023), 135.

\bibitem{tan-2018-decay-bound}
Z. Tan and Y. Wang, {\em Global well-posedness of an initial-boundary value problem for viscous nonresistive MHD systems}, SIAM J. Math. Anal. 50 (2018) 1432--1470.

\bibitem{tao}
T. Tao, {\em Nonlinear dispersive equations: local and global analysis} (CBMS regional conference series in mathematics) (Providence, RI: American Mathematical Society)

\bibitem{TYZ} C.V. Tran, X. Yu and Z. Zhai, {\em On global regularity of 2D generalized magnetohydrodynamic equations}, J. Differ. Equ. 254 (2013),  4194--4216.

\bibitem{WU1}J. Wu, {\em Generalized MHD equations}, J. Differ. Equ. 195 (2002),  284--312.
\bibitem{WU2}J. Wu, {\em Global regularity for a class of generalized magnetohydrodynamic equations}, J. Math. Fluid Mech. 13 (2011), 295--305.

\bibitem{zhang-2020-GlobalWellPosedness2D}
D. Wei and Z. Zhang, {\em Global well-posedness for the 2-D MHD equations with magnetic diffusion}, Commun. Math. Res. 36 (2020), no. 4, 377--389.



\bibitem{xu-2015-GlobalSmallSolutions}
L. Xu and P. Zhang, {\em Global small solutions to three-dimensional incompressible magnetohydrodynamical system}, SIAM J. Math. Anal. 47 (2015), no. 1, 26--65.

\bibitem{YS}
H. Yu and  H. Shang, {\em Global existence and asymptotics for the nD generalized MHD  equations}, submitted, 2024.

	\end{thebibliography}
	\end{document}